\documentclass[12pt]{amsart}
\usepackage{amsthm,amsopn,amsfonts,amssymb,pb-diagram}
\usepackage{tabmac}
\usepackage{fullpage}

\usepackage{euler, amsfonts, amssymb, latexsym, epsfig, epic}
\usepackage{epstopdf}

\usepackage[backref]{hyperref}
   \renewcommand*{\backrefalt}[4]{%
     \ifcase #1 %
     \or
     \vspace{0pt} \footnotesize (In \S#2)%
     \else
     \vspace{0pt} \footnotesize (In \S#2)
     \fi}
\renewcommand*{\backref}[1]{}

\font\co=lcircle10
\def\boxcross{\ \smash{\lower6.5pt\hbox{\rlap{\hskip4.5pt\vrule height13.5pt}}
                \raise0pt\hbox{\rlap{\hskip-2pt \vrule height.4pt depth0pt
                        width13.5pt}}}\hskip12.7pt}

\def\boxelbow{\ \hskip.1pt\smash{%
               \hbox{\co \hskip 5.5pt\rlap{\mathsurround=0pt\rlap{\mathsurround=0pt\char'006}\lower0.4pt\rlap{\char'004}}
                \lower6.5pt\rlap{\hskip-0.2pt\vrule height3pt}
                \raise3.5pt\rlap{\hskip-0.2pt\vrule height3.2pt}}
                \hbox{%
                  \rlap{\hskip-6.4pt \vrule height.4pt depth0pt
width2.5pt}%
                  \rlap{\hskip4.05pt \vrule height.4pt depth0pt
width3.1pt}}}
                \hskip8.7pt}

\newtheorem{theorem}{Theorem}[section]
\newtheorem{thm}[theorem]{Theorem}
\newtheorem*{thm*}{Theorem}
\newtheorem{lemma}[theorem]{Lemma}
\newtheorem{lem}[theorem]{Lemma}

\newtheorem{prop}[theorem]{Proposition}

\newtheorem{conjecture}[theorem]{Conjecture}

\newtheorem{cor}[theorem]{Corollary}

\theoremstyle{definition}

\theoremstyle{remark}
\newtheorem{example}[theorem]{Example}
\newtheorem{remark}[theorem]{Remark}
\newcommand{\Xo}{\mathring{X}}
\newcommand{\Pio}{\mathring{\Pi}}
\newcommand{\SR}{\mathrm{SR}}
\newcommand{\In}{\mathrm{In}}
\newcommand{\AFl}{\widetilde{{\mathcal F} \ell}}

\newcommand{\C}{{\mathbb C}}
\newcommand{\Gr}{\mathrm{Gr}}
\newcommand{\GS}{S_{n,k}^{\mathrm{min}}}
\newcommand{\AGS}{S_{n,k}^{\mathrm{max}}}

\def\id{\rm id}
\newcommand{\JJ}{G}
\newcommand{\J}{{\mathcal J}}
\newcommand{\M}{{\mathcal M}}
\newcommand{\pt}{\mathrm{pt}}
\newcommand{\Q}{{\mathcal Q}}
\newcommand{\R}{{\mathbb R}}
\newcommand{\tF}{\tilde F}

\newcommand{\tS}{\tilde S}
\newcommand{\weak}{\mathrm{weak}}
\newcommand{\Z}{\mathbb Z}
\newcommand{\onto}{\twoheadrightarrow}
\newcommand{\CC}{\mathbb{C}}

\newcommand{\ZZ}{\mathbb{Z}}
\newcommand{\PP}{\mathbb{P}}

\newcommand\defn[1]{{\bf #1}}

\DeclareMathOperator{\Spec}{Spec}

\DeclareMathOperator{\av}{av} \DeclareMathOperator{\Par}{Par}
\DeclareMathOperator{\Bound}{Bound}
 \DeclareMathOperator{\GGMS}{GGMS}
\DeclareMathOperator{\Jugg}{Jugg} \DeclareMathOperator{\Span}{Span}
\newcommand{\Fl}{F \ell}
\renewcommand\O{{\mathcal O}}

\begin{document}

\title{Positroid varieties: juggling and geometry}
\author{Allen Knutson}
\address{Department of Mathematics, Cornell University, Ithaca, NY 14853 USA}
\email{allenk@math.cornell.edu}
\thanks{AK was partially supported by NSF grant DMS-0604708.}
\author{Thomas Lam}
\address{Department of Mathematics, University of Michigan, Ann Arbor, MI 48109 USA}
\email{tfylam@umich.edu}
\thanks{TL was partially supported by NSF grants DMS-0600677,
  DMS-0652641, and DMS-0901111, and by a Sloan Fellowship.}
\author{David E Speyer}
\address{Department of Mathematics, University of Michigan, Ann Arbor, MI 48109 USA}
\email{speyer@umich.edu}
\thanks{DES was supported by a Research Fellowship from the Clay Mathematics Institute}
\date{\today}

\begin{abstract}
  While the intersection of the Grassmannian Bruhat decompositions for
  all coordinate flags is an intractable mess, the intersection of
  only the {\em cyclic shifts} of one Bruhat decomposition turns out to have
  many of the good properties of the Bruhat and Richardson decompositions.

  This decomposition coincides with the projection of the Richardson
  stratification of the flag manifold, studied by Lusztig, Rietsch, Brown-Goodearl-Yakimov and the present authors. However, its cyclic-invariance is hidden in
  this description. Postnikov gave many cyclic-invariant ways to index
  the strata, and we give a new one, by a subset of the affine Weyl
  group we call {\em bounded juggling patterns}. We call the strata {\em positroid varieties.}

  Applying results from \cite{KLS}, we show that positroid varieties
  are normal, Cohen-Macaulay, have rational singularities, and are
  defined as schemes by the vanishing of Pl\"ucker coordinates.  We
  prove that their
  associated cohomology classes are represented by affine Stanley functions.
  This latter fact lets us connect Postnikov's and Buch-Kresch-Tamvakis'
  approaches to quantum Schubert calculus.
%
\end{abstract}

\maketitle

\setcounter{tocdepth}{1}
{\footnotesize \tableofcontents}

\section{Introduction, and statement of results}

\subsection{Some decompositions of the Grassmannian}\label{ssec:decomps}

This paper is concerned with the geometric properties of a stratification of the Grassmannian studied in
\cite{Lus,Pos,Rie,BGY,Wil}. It fits into a family of successively
finer decompositions:
$$ \{\text{Bruhat cells}\} ,
\{\text{open Richardson varieties}\} ,
\{\text{open \bf positroid varieties}\} ,
\{\text{GGMS strata}\}. $$
We discuss the three known ones in turn, and then see how the
family of positroid varieties fits in between.

The {\em Bruhat decomposition} of the Grassmannian of $k$-planes in
$n$-space dates back, despite the name, to Schubert in the 19th
century. It has many wonderful properties:
\begin{itemize}
\item the strata are easily indexed (by partitions in a $k\times (n-k)$ box)
\item it is a stratification: the closure (a {\em Schubert variety})
  of one open stratum is a union of others
\item each stratum is smooth and irreducible (in fact a cell)
\item although the closures of the strata are (usually) singular,
  they are not too bad: they are normal and Cohen-Macaulay, and
  even have rational singularities.
\end{itemize}

The Bruhat decomposition is defined relative to a choice of coordinate flag,
essentially an ordering on the basis elements of $n$-space.
The {\em Richardson decomposition} is the common refinement of the
Bruhat decomposition and the {\em opposite} Bruhat decomposition, using
the opposite order on the basis. Again, many excellent properties hold for
this finer decomposition:
\begin{itemize}
\item it is easy to describe the nonempty intersections of Bruhat and opposite Bruhat
  strata (they correspond to {\em nested} pairs of partitions)
\item it is a stratification, each open stratum is smooth and irreducible,
  and their closures are normal and Cohen-Macaulay
  with rational singularities \cite{BrionPos}.
\end{itemize}

At this point one might imagine intersecting the Bruhat decompositions
relative to {\em all} the coordinate flags, so as not to prejudice one
over another. This gives the {\em GGMS decomposition} of the Grassmannian
\cite{GGMS},
and as it turns out, these good intentions pave the road to Hell:
\begin{itemize}
\item it is infeasible to index the nonempty strata~\cite{LostAxiom}
\item it is not a stratification \cite[\S 5.2]{GGMS}
\item the strata can have essentially any singularity~\cite{Mnev}.
  In particular, the nonempty ones need not be irreducible,
  or even equidimensional.
\end{itemize}

This raises the question: can one intersect more than two permuted Bruhat
decompositions, keeping the good properties of the Bruhat and
Richardson decompositions, without falling into the GGMS abyss?

The answer is yes: we will intersect the $n$ {\em cyclic
permutations} of the Bruhat decomposition. That is to say, we will define 
an {\em open positroid variety} to be an intersection of $n$ Schubert cells, taken with respect to the 
the $n$ cyclic rotations of the standard flag. We will define a {\em positroid variety} to be
the closure of an open positroid variety. See section~\ref{sec:pos} for details.

It is easy to show,
though not immediately obvious, that this refines the Richardson
decomposition. It is even less obvious, though also true, that the
open positroid varieties are smooth and irreducible (as we discuss in Section
\ref{ssec:projectedRichardsons}).

There is a similar decomposition for any partial flag manifold $G/P$,
the projection of the Richardson stratification from $G/B$.
That decomposition arises in the study
of several seemingly independent structures:
\begin{itemize}
\item total nonnegativity, in e.g. \cite{Lus,Pos,Rie},
  see \S \ref{ssec:jugafftnn};
\item prime ideals in noncommutative deformations of $G/P$ (though worked
  out only for the Grassmannian, in \cite{LLR}),
  and a semiclassical version thereof in Poisson geometry \cite{BGY,GY};
\item the characteristic $p$ notion of Frobenius splitting
  (\cite{KLS}).
\end{itemize}

We show that the positroid stratification and the projected Richardson stratification coincide. Specifically, we prove:

\begin{thm*}[Theorem \ref{thm:projectedRichardsons}]
  If $X_u^w$ is a Richardson variety in the full flag manifold
  ($u,w \in S_n$), then its image under projection to $\Gr(k,n)$
  is a positroid variety. If $w$ is required to be a Grassmannian
  permutation, then every positroid variety arises uniquely this way.
\end{thm*}

Theorem~\ref{thm:projectedRichardsons} has been suspected, but has not previously been proved in print, and is surprisingly difficult in its details. 
This result was already known  on the positive part of $\Gr(k,n)$, as we explain in Remark~\ref{rem:StratHistory}.

Once we know that positroid varieties are projected Richardson varieties, the following geometric properties follow from the results of \cite{KLS}.  Part (1) of the following Theorem was also established by Billey and Coskun \cite{BC} for projected Richardson varieties.

\begin{thm*}[\cite{KLS} and Theorem \ref{T:LinearGeneration}]
\ 
  \begin{enumerate}
  \item Positroid varieties are normal and Cohen-Macaulay,
    with rational singularities.
  \item Though positroid varieties are defined as the closure of
    the intersection of $n$ cyclically permuted Bruhat cells,
    they can also be defined (even as schemes) as
    the intersection of the $n$ cyclically permuted Schubert varieties.
    In particular, each positroid variety is defined as a scheme by the
    vanishing of some Pl\"ucker coordinates.
  \end{enumerate}
\end{thm*}

\begin{thm*}[\cite{KLS}]
  The standard Frobenius spliting on the Grassmannian compatibly
  splits all positroid varieties.  Furthermore, positroid varieties
  are exactly the compatibly split subvarieties of the Grassmannian.
\end{thm*}

\newcommand\into{\hookrightarrow}

Before going on, we mention a very general construction given
two decompositions $\{Y_a\}_{a\in A}$,$\{Z_b\}_{b\in B}$ of
a scheme $X$, one refining the other. Assume that
\begin{itemize}
\item $X = \coprod_A Y_a = \coprod_B Z_b$,
\item for each $a\in A$, there exists a subset $B_a \subseteq B$ such that
  $Y_a = \coprod_{B_a} Z_b$,
\item each $Y_a$ is irreducible (hence nonempty), and
  each $Z_b$ is nonempty. (We do not assume that each $Z_b$ is irreducible.)
\end{itemize}
Then there is a natural surjection $B \onto A$ taking $b$ to the unique
$a$ such that $Z_b \subseteq Y_a$, and a natural inclusion $A\into B$
taking $a$ to the unique $b\in B_a$ such that $Z_b$ is open in $Y_a$.
(Moreover, the composite $A \into B \onto A$ is the identity.)
We will call the map $B\onto A$ the \defn{$A$-envelope}, and will generally use
the inclusion $A\into B$ to identify $A$ with its image.
Post this identification, each $a\in A$ corresponds to two strata $Y_a$, $Z_a$,
and we emphasize that these are usually {\em not} equal;
rather, one only knows that $Y_a$ contains $Z_a$ densely.

To each GGMS stratum $X$, one standardly associates the set of
coordinate $k$-planes that are elements of $\overline X$,
called the \defn{matroid of $X$}.
(While ``matroid'' has many simple definitions, this is not one of
them; only \defn{realizable} matroids arise this way, and
characterizing them is essentially out of reach \cite{LostAxiom}.)
It is a standard, and easy, fact that the matroid characterizes
the stratum, so via the $A\into B$ yoga above,
we can index the strata in the Schubert, Richardson,
and positroid decompositions by special classes of matroids.
Schubert matroids have been rediscovered many times in the matroid
literature (and renamed each time; see \cite{Bonin}).
Richardson matroids are known as {\em lattice path matroids} \cite{Bonin}.
The matroids associated to the open positroid varieties are
exactly the positroids \cite{Pos} (though Postnikov's original definition
was different, and we give it in the next section).

In our context, the observation two paragraphs above says that if a
matroid $M$ is a positroid, then the positroid stratum of $M$ is
usually {\em not} the GGMS stratum of $M$, but only contains it
densely.

\begin{remark}\label{rem:SashaParam}
  For each positroid $M$, Postnikov gives many parametrizations by
  $\R_{+}^{\ell}$ of the totally nonnegative part (whose definition we will
  recall in the next section) of the GGMS stratum of $M$.  Each
  parametrization extends to a rational map $(\C^\times)^{\ell} \to \Gr(k,n)$;
  if we use the parametrization coming (in Postnikov's terminology)
  from the Le-diagram of $M$ then this map is well defined on all of
  $(\C^\times)^{\ell}$.  The image of this map is neither the GGMS
  stratum nor the positroid stratum of $M$ (although the nonnegative parts of all three coincide). For example, if
  $(k,n)=(2,4)$ and $M$ is the ``uniform'' matroid in which any two elements of
  $[4]$ are independent, this parametrization is
  $$(a,b,c,d) \mapsto (p_{12}: p_{13}: p_{14}: p_{23}: p_{24}: p_{34})
  = (1: d : cd: bd: (a+1)bcd: abcd^2).$$
  The image of this map is the open set where $p_{12}$, $p_{13}$,
  $p_{14}$, $p_{23}$ and $p_{34}$ are nonzero. It is smaller than the
  positroid stratum, where $p_{13}$ can be zero. The image is larger than the GGMS stratum, where $p_{24}$ is  also nonzero. 
  
  One may regard this, perhaps, as evidence that matroids
  are a philosophically incorrect way to index the strata. We shall see
  another piece of evidence in Remark \ref{rem:notPluckerDefined}.
\end{remark}

\newcommand\Sym{{\rm Sym}}
\newcommand\integers{\Z}
\newcommand\naturals{{\mathbb N}}

\subsection{Juggling patterns, affine Bruhat order, and total nonnegativity}\label{ssec:jugafftnn}

We now give a lowbrow description of the decomposition we are
studying, from which we will see a natural indexing of the strata.

\newcommand\junk[1]{}

Start with a $k\times n$ matrix $M$ of rank $k$ ($\leq n$),
and think of it as a list of column vectors $\vec v_1,\ldots,\vec v_n$. 
Extend this to an infinite but repeating list 
$\ldots,\vec v_{-1}, \vec v_0, \vec v_1,\ldots, \vec v_n, \vec v_{n+1}, \ldots$
where $\vec v_i = \vec v_j$ if $i\equiv j \bmod n$. 
Then define a function $f:\ZZ \to \ZZ$ by 
$$ f(i) = \min\ \left\{ j \geq i : \vec v_i \in 
{\rm span}( \{\vec v_{i+1}, \vec v_{i+2}, \ldots, \vec v_{j}) \right\} $$
Since $\vec v_{n+i} = \vec v_i$, each $f(i) \leq n+i$, and each
$f(i) \geq i$ with equality only if $\vec v_i = \vec 0$. 
It is fun to prove that $f$ must be $1:1$, and has enough finiteness
to then necessarily be onto as well.
Permutations of $\ZZ$ satisfying $f(i+n) = f(i) + n\ \forall i$
are called \defn{affine permutations}, and the group thereof can
be identified with the affine Weyl group of $GL_n$ (see e.g. \cite{ER}).

This association of an affine permutation to each $k\times n$ matrix of rank $k$
depends only on the $k$-plane spanned by the rows, and so descends
to $\Gr(k,n)$, where it provides a complete combinatorial invariant of the
strata in the cyclic Bruhat decomposition. 

\begin{thm*}[Theorem \ref{T:pairsboundedposet}, Corollary \ref{cor:shell}] \footnote{This result is extended to projected Richardson varieties in partial flag varieties $G/P$ of arbitrary type by He and Lam \cite{HL}.}
  This map from the set of positroid strata to the affine Weyl group is
  order-preserving, with respect to the closure order on positroid strata
  (Postnikov's \defn{cyclic Bruhat order}) and the affine Bruhat order,
  and identifies the set of positroids with a downward Bruhat order ideal.

  Consequently, the cyclic Bruhat order is Eulerian and EL-shellable
  (as shown by hand already in \cite{Wil}).
\end{thm*}

We interpret these $f$ physically as follows. Consider a juggler who is
juggling $k$ balls, one throw every second, doing a pattern of period $n$.
At time $i \in \ZZ$, they throw a ball that lands shortly\footnote{%
  almost exactly at time $f(i)-\frac{1}{2}$, according to video
  analysis of competent jugglers} 
before, to be thrown again at, time $f(i)$. 
No two balls land at the same time, and there is always a ball 
available for the next throw.
If we let $t_i = f(i)-i$ be the \defn{throw} at time $i$,
this cyclic list of $n$ numbers $(t_1,\ldots,t_n)$ is a {\em juggling pattern}%
\footnote{%
  Not every juggling pattern arises this way;
  the patterns that arise from matrices can only have throws
  of height $\leq n$. This bound is very unnatural from the juggling
  point of view, as it excludes the standard $3$-ball cascade
  $(t_1 = 3)$ with period $n=1$.}
or {\em siteswap} (for which our references are \cite{Polster,Siteswap};
see also \cite{BuhlerEGW,ER,Warrington,G1,G2}).
This mathematical model of juggling was developed by several groups of
jugglers independently in 1985, and is of great practical use
in the juggling community.

If $M$ is generic, then 
the pattern is the lowest-energy pattern, where every throw
is a $k$-throw.\footnote{%
  These juggling patterns are called ``cascades'' for $k$ odd and
  ``(asynchronous) fountains'' for $k$ even.}
At the opposite extreme, imagine that $M$ only has entries in
some $k$ columns. Then $n-k$ of the throws are $0$-throws,
and $k$ are $n$-throws.\footnote{%
  These are not the most excited $k$-ball patterns of length $n$;
  those would each have a single $kn$-throw, all the others being $0$-throws.
  But juggling patterns associated to matrices must have each $t_i \leq n$.}

If one changes the cyclic action slightly, by moving the first column
to the end {\em and multiplying it by} $(-1)^{k-1}$, then one
preserves the set of real matrices for which every $k\times k$
submatrix has nonnegative determinant. This, by definition, lies over the
\defn{totally nonnegative part} $\Gr(k,n)_{\geq 0}$
of the Grassmannian. (This action may have period either $n$ or $2n$ up on matrices, but it always has period $n$ down on the Grassmannian.)
Postnikov's motivation was to describe those matroids whose GGMS strata 
intersect this totally nonnegative part; it turns out that
they are exactly the positroids, and the totally nonnegative part
of each open positroid stratum is homeomorphic to a ball.


\begin{remark}\label{rem:StratHistory}
Now that we have defined the totally nonnegative part of the Grassmannian, we can explain
the antecedents to Theorem~\ref{thm:projectedRichardsons}.
Postnikov~(\cite{Pos}) defined  the totally nonnegative part of the Grassmannian as we have done above, by nonnegativity of all minors.
Lusztig~(\cite{Lus}) gave a different definition which applied to any $G/P$.
That the two notions agree is not obvious, and was established in~\cite{Rie09}.
In particular, the cyclic symmetry seems
to be special to Grassmannians.\footnote{Milen Yakimov has
  proven the stronger result that the standard Poisson
  structure on $\Gr(k,n)$, from which the positroid stratification
  can be derived, is itself cyclic-invariant \cite{Milen}.}

Lusztig, using his definition, gave a stratification of $(G/P)_{\geq 0}$ by the projections of Richardson varieties. 
Theorem~3.8 of \cite{Pos} (which relies on the results of~\cite{MR}
and~\cite{RW}) states that Postnikov's and Lusztig's stratifications
of $\Gr(k,n)_{\geq 0}$ coincide. 
This result says nothing about how the stratifications behave away from the totally nonnegative region.
Theorem~\ref{thm:projectedRichardsons} can be thought of as a
complex analogue of~\cite[Theorem~3.8]{Pos}; it implies but does not
follow from~\cite[Theorem~3.8]{Pos}.

We thank Konni Rietsch for helping us to understand the connections between these results.
\end{remark}

\subsection{Affine permutations, and the associated cohomology class of a positroid variety}

\newcommand\iso{\cong}
\newcommand\complexes{{\mathbb C}}
\newcommand\codim{{\rm codim\ }}

Given a subvariety $X$ of a Grassmannian, one can canonically
associate a symmetric polynomial in $k$ variables, in a couple of
equivalent ways:
\begin{enumerate}
\item
  Sum, over partitions $\lambda$ with $|\lambda| = \codim X$,
  the Schur polynomial $S_\lambda(x_1,\ldots,x_k)$
  weighted by the number of points of intersection of $X$ with a
  generic translate of $X_{\lambda^c}$ (the Schubert variety associated
  to the complementary partition inside the $k\times (n-k)$ rectangle).
\item
  Take the preimage of $X$ in the Stiefel manifold of $k\times n$ matrices
  of rank $k$, and the closure $\overline X$ inside $k\times n$ matrices.
  (In the $k=1$ case this is the affine cone over a projective variety,
  and it seems worth it giving the name ``Stiefel cone'' in general.)
  This has a well-defined class in the equivariant Chow ring
  $A^*_{GL(k)}(\complexes^{k\times n})$,
  which is naturally the ring of symmetric polynomials in $k$ variables.
\end{enumerate}
The most basic case of $X$ is a Schubert variety $X_\lambda$, in which case 
these recipes give the Schur polynomial $S_\lambda$.
More generally, the first construction shows that the symmetric polynomial
must be ``Schur-positive'', meaning a positive sum of Schur polynomials.

In reverse, one has ring homomorphisms
$$ \{\text{symmetric functions}\}
\onto \integers[x_1,\ldots,x_k]^{S_k}
\iso A^*_{GL(k)}(\complexes^{k\times n})
\onto A^*_{GL(k)}(\text{Stiefel})
\iso A^*(\Gr(k,n)) $$
and one can ask for a symmetric function $f$ whose image is the class $[X]$.

\begin{thm*}[Theorem \ref{thm:affineStanley}\footnote{
Snider \cite{Snider} has given a direct geometric explanation of this result by identifying affine patches on $\Gr(k,n)$ with opposite Bruhat cells in the affine flag manifold,
in a way that takes the positroid stratification to the Bruhat decomposition.  Also, an analogue of this result for projected Richardson varieties in an arbitrary $G/P$ is established by He and Lam \cite{HL}: the connection with symmetric functions is absent, but the cohomology classes of projected Richardson varieties and affine Schubert varieties are compared via the affine Grassmannian.
}]
  The cohomology class associated to a positroid variety can be
  represented by the affine Stanley function of its affine permutation,
  as defined in \cite{Lam1}.
\end{thm*}

This is a surprising result in that affine Stanley functions are not
Schur-positive in general, even for this restricted class of affine
permutations. Once restricted to the variables $x_1,\ldots,x_k$, they are!
In Theorem \ref{thm:positiveKTclass}
we give a much stronger abstract positivity result,
for positroid classes in $T$-equivariant $K$-theory.

Our proof of Theorem \ref{thm:affineStanley} is inductive.
In future work, we hope to give a direct geometric proof of this and
Theorem \ref{T:pairsboundedposet}, by embedding the Grassmannian
in a certain subquotient of the affine flag manifold,
and realizing the positroid decomposition as the transverse pullback
of the affine Bruhat decomposition.

\subsection{Quantum cohomology 
  and toric Schur functions}
%
%
%

In \cite{BKT}, Buch, Kresch, and Tamvakis related quantum Schubert
calculus on Grassmannians to ordinary Schubert calculus on $2$-step
partial flag manifolds.  In \cite{PosQH}, Postnikov showed that the
structure constants of the quantum cohomology of the Grassmannian
were encoded in symmetric functions he called toric Schur
polynomials. We connect these ideas to positroid varieties:

\begin{thm*}[Theorem \ref{thrm:QuantumPositroid}]
Let $S \subset \Gr(k,n)$ be the union of all genus-zero stable
curves of degree $d$ which intersect a fixed Schubert variety $X$
and opposite Schubert variety $Y$.  Suppose there is a non-trivial
quantum problem associated to $X,Y$ and $d$.  Then $S$ is a
positroid variety: as a projected Richardson variety it is obtained
by a pull-push from the 2-step flag variety considered in
\cite{BKT}. Its cohomology class is given by the toric Schur
polynomial of \cite{PosQH}.
\end{thm*}

The last statement of the theorem is consistent with the connection
between affine Stanley symmetric functions and toric Schur functions
(see \cite{Lam1}).

\subsection*{Acknowledgments}

Our primary debt is of course to Alex Postnikov, for getting us
excited about positroids. We also thank Michel Brion, Leo Mihalcea, Su-Ho Oh, Konni Rietsch,
Frank Sottile, Ben Webster, Lauren Williams, and Milen Yakimov for useful conversations.

\section{Some combinatorial background}

Unless otherwise specified, we shall assume that nonnegative
integers $k$ and $n$ have been fixed, satisfying $0 \leq k \leq n$.

\subsection{Conventions on partitions and permutations}
\label{ssec:conventions}
For integers $a$ and $b$, we write
$[a,b]$ to denote the interval $\{a, a+1, \ldots, b \}$, and $[n]$
to denote the initial interval $\{ 1,2, \ldots, n \}$.
If $i \in \Z$, we let $\bar i \in [n]$ be the unique integer
satisfying $i \equiv \bar i \mod n$.  We write
$\binom{S}{k}$ for the set of $k$-element subsets of $S$.  Thus
$\binom{[n]}{k}$ denotes the set of $k$-element subsets of
$\{1,2,\ldots,n\}$.  

As is well known, there is a bijection between $\binom{[n]}{k}$ and
the partitions of $\lambda$ contained in a $k \times (n-k)$ box.
There are many classical objects, such as Schubert varieties, which
can be indexed by either of these $\binom{n}{k}$-element sets.  We
will favor the indexing set $\binom{[n]}{k}$, and will only discuss
the indexing by partitions when it becomes essential, in
\S \ref{sec:cohomology}.

%

We let $S_n$ denote the permutations of the set $[n]$.  A
permutation $w \in S_n$ is written in one-line notation as
$[w(1)w(2)\cdots w(n)]$.  Permutations are multiplied from right to
left so that if $u, w \in S_n$, then $(uw)(i) = u(w(i))$.  Thus
multiplication on the left acts on values, and multiplication on the
right acts on positions.  Let $w \in S_n$ be a permutation.  An
\defn{inversion} of $w$ is a pair $(i,j) \in [n] \times [n]$ such that $i <
j$ and $w(i) > w(j)$.  The length $\ell(w)$ of a permutation $w \in
S_n$ is the number of its inversions.  A factorization $w = uv$ is
called \defn{length-additive} if $\ell(w) = \ell(u) + \ell(v)$.

The longest element $[n(n-1)\cdots 1]$ of $S_n$ is denoted $w_0$.
The permutation $[234 \cdots n 1]$ is denoted $\chi$ (for Coxeter
element).  As a Coxeter
group, $S_n$ is generated by the simple transpositions $\{s_i = [12
\cdots(i-1)(i+1)i(i+2)\cdots n]\}$.

For $k \in [0,n]$, we let $S_k \times S_{n-k} \subseteq S_n$ denote
the parabolic subgroup of permutations which send $[k]$ to $[k]$ and
$[k+1,n]$ to $[k+1,n]$.  A permutation $w \in S_n$ is called
\defn{Grassmannian} (resp. \defn{anti-Grassmannian}) if it is
minimal (resp. maximal) length in its coset $w(S_k \times S_{n-k})$;
the set of such permutations is denoted $\GS$ (resp. $\AGS$).

If $w \in S_n$ and $k \in [0,n]$, then $\sigma_k(w) \in \binom{[n]}{k}$ 
denotes the set $w([k])$.
Often, we just write $\sigma$ for $\sigma_k$ when no confusion will arise.  The map
$\sigma_k: S_n \to \binom{[n]}{k}$ is a bijection when restricted to $\GS$.

\subsection{Bruhat order and weak order}
We define a partial order $\leq$ on $\binom{[n]}{k}$ as
follows.  For $I = \{i_1 < i_2 < \cdots < i_k\}$ and $J = \{j_1 < j_2
\cdots < j_k\} \in \binom{[n]}{k}$, we write $I \leq J$ if $i_r \leq
j_r$ for $r \in [k]$.

We shall denote the \defn{Bruhat order}, also called the
\defn{strong order}, on $S_n$ by $\leq$ and $\geq$.  One has the
following well known criterion for comparison in Bruhat order: if
$u, w \in S_n$ then $u \leq w$ if and only if $u([k]) \leq w([k])$
for each $k \in [n]$.  Covers in Bruhat order will be denoted by
$\lessdot$ and $\gtrdot$.
The map $\sigma_k: (\GS, \leq) \to \left( \binom{[n]}{k},\leq\right)$
is a poset isomorphism.

The \defn{(left) weak order} $\leq_{\weak}$ on $S_n$ is the
transitive closure of the relations
$$
w \leq_{\weak} s_iw \ \ \ \mbox{if $\ell(s_iw) = \ell(w) + 1$}.
$$
The weak order and Bruhat order agree when restricted to $\GS$.

\subsection{$k$-Bruhat order and the poset $\Q(k,n)$.}
\label{sec:kBruhat}

The \defn{$k$-Bruhat order} \cite{BS,LS} $\leq_k$ on $S_n$ is
defined as follows. Let $u$ and $w$ be in $S_n$.  Then $u$ \defn{$k$-covers}
$w$, written $u \gtrdot_k w$, if and only if $u \gtrdot w$ and
$\sigma_{k} (u) \neq \sigma_{k} (w)$.  The $k$-Bruhat order is the partial order on
$S_n$ generated by taking the transitive closure of these cover
relations (which remain cover relations). We let $[u,w]_k \subset
S_n$ denote the interval of $S_n$ in $k$-Bruhat order.  It is shown
in \cite{BS} that every interval $[u, w]_k$ in $(S_n, \leq_k)$ is
a graded poset with rank $\ell(w)- \ell(u)$.  We have the following
criterion for comparison in $k$-Bruhat order.

\begin{thm}[{\cite[Theorem A]{BS}}]
\label{T:BScriterion} Let $u, w \in S_n$.  Then $u \leq_k w$ if and
only if
\begin{enumerate}
\item $1 \leq a \leq k < b \leq n$ implies $u(a) \leq w(a)$ and $u(b)
\geq w(b)$.
\item If $a < b$, $u(a) < u(b)$, and $w(a) > w(b)$, then $a \leq k <
b$.
\end{enumerate}
\end{thm}

Define an equivalence relation on the set of $k$-Bruhat intervals,
generated by the relations that $[u,w]_k \sim [x,y]_k$ if there is a
$z \in S_k \times S_{n-k}$ so that we have length-additive
factorizations $uz = x$ and $wz = y$.
If $u \leq_k w$, we let $\langle u,w \rangle$ denote the equivalence
class containing $[u,w]_k$.  Let $\Q(k,n)$ denote the equivalence
classes of $k$-Bruhat intervals.

We discuss this construction in greater generality in \cite[\S 2]{KLS}.
To obtain the current situation, specialize the results of that paper to $(W, W_P) = (S_n, S_{k} \times S_{n-k})$. 
The results we describe here are all true in that greater generality.

\begin{prop} \label{prop:EasyEquiv}
If $[u_1, w_1] \sim [u_2, w_2]$ then $u_1^{-1} u_2 = w_1^{-1} w_2$ and this common ratio is in $S_k \times S_{n-k}$. Also $\ell(w_1) - \ell(w_2) = \ell(u_1) - \ell(u_2)$.
\end{prop}

\begin{proof}
This is obvious for the defining equivalences and is easily seen to follow for a chain of equivalences. \end{proof}

We will prove a converse of this statement below as Proposition~\ref{prop:Equivalence}.
The reader may prefer this definition of $\sim$.

\begin{prop} \label{prop:GrassRep}
Every equivalence class in $\Q(k,n)$ has a unique representative of the form $[u',w']$ where $w'$ is Grassmannian.  
If $[u,w]$ is a $k$-Bruhat interval, and $[u', w']$ is equivalent to $[u,w]$ with $w'$ Grassmannian, then we have length-additive factorizations $u=u' z$ and $w = w' z$ with $z \in S_k \times S_{n-k}$.
\end{prop}

\begin{proof}
See \cite[Lemma 2.4]{KLS} for the existence of a representative of this form. 
If $[u', w']$ and $[u'', w'']$ are two such representatives, then $(w')^{-1} w''$ is in $S_k \times S_{n-k}$ and both $w'$ and $w''$ are Grassmannian, so $w'=w''$. Then $(u')^{-1} u'' = (w')^{-1} w'' = e$ so $u' = u''$ and we see that the representative is unique.

Finally, let $[u', w']$ be the representative with $w'$ Grassmannian, and let $[u,w] \sim [u', w']$.
Set $z = (u')^{-1} u = (w')^{-1} w$ with $z \in S_k \times S_{n-k}$. 
Since $w'$ is Grassmannian, we have $\ell(w) = \ell(w') + \ell(z)$. 
Then the equation  $\ell(w) - \ell(w') = \ell(u) - \ell(u')$ from Proposition~\ref{prop:EasyEquiv} shows that $\ell(u) = \ell(u') + \ell(z)$ as well.
So the products $u=u'z$ and $w=w'z$ are both length-additive, as desired.
\end{proof}

We can use this observation to prove a more computationally useful version of the equivalence relation:
\begin{prop} \label{prop:Equivalence}
Given two $k$-Bruhat intervals $[u_1, w_1]$ and $[u_2, w_2]$, we have $[u_1, w_1] \sim [u_2, w_2]$ if and only if $u_1^{-1} u_2 = w_1^{-1} w_2$ and common ratio lies in $S_k \times S_{n-k}$.
\end{prop}

\begin{proof}
The forward implication is Proposition~\ref{prop:EasyEquiv}. 
For the reverse implication, let $[u_1, w_1]$ and $[u_2, w_2]$ be as stated.
Let $[u'_1, w'_1]$ and $[u'_2, w'_2]$ be the representatives with $w'_i$ Grassmannian. 
Since $(w'_1)^{-1} w'_2$ is in $S_k \times S_{n-k}$, and both $w'_i$ are Grassmannian, then $w'_1 = w'_2$.
Since $(u'_1)^{-1} u'_2 = (w'_1)^{-1} w'_2 = e$, we deduce that $u'_1 = u'_2$.
So $[u_1, w_1] \sim [u'_1, w'_1] = [u'_2, w'_2] \sim [u_2, w_2]$ and we have the reverse implication.
\end{proof}

We also cite:

%
%

\begin{thm}[{\cite[Theorem 3.1.3]{BS}}]
\label{T:BSisom} If $u \leq_k w$ and $x \leq_k y$ with $wu^{-1} =
yx^{-1}$, then the map $v \mapsto vu^{-1}x$ induces an isomorphism
of graded posets $[u,w]_k \to [x,y]_k$.
\end{thm}

We equip $\Q(k,n)$ with a partial order $\leq$ given by $q' \leq q$
if and only if there are representatives $[u,w]_k \in q$ and
$[u',w']_k \in q'$ so that $[u', w'] \subseteq [u, w]$.  This partial order was studied by Rietsch \cite{Rie}, see also \cite{Wil,GY}.

The poset $\Q(2,4)$ already has $33$ elements; its Hasse diagram
appears in \cite{Wil}.  See also Figure \ref{fig:gr24}.

\section{Affine permutations, juggling patterns and positroids}
\label{sec:posets}

Fix integers $0 \leq k \leq n$. In this section, we will define several posets of objects and prove that the posets are all isomorphic. We begin by surveying the posets we will consider. The objects in these posets will index positroid varieties, and all of these indexing sets are useful. All the isomorphisms we define are compatible with each other. Detailed definitions, and the definitions of the isomorphisms, will be postponed until later in the section.

We have already met one of our posets, the poset $\Q(k,n)$ from \S \ref{sec:kBruhat}.

The next poset will be the poset $\Bound(k,n)$ of bounded affine permutations:
these are bijections $f: \ZZ \to \ZZ$ such that $f(i+n)=f(i)+n$,
$i \leq f(i) \leq f(i)+n$ and $(1/n) \sum_{i=1}^n (f(i)-i) =k$. After
that will be the poset $\Jugg(k,n)$ of bounded juggling patterns. The elements
of this poset are $n$-tuples $(J_1, J_2, \ldots, J_n) \in \binom{[n]}{k}^n$
such that $J_{i+1} \supseteq (J_i \setminus \{ 1 \}) - 1$, where the
subtraction of $1$ means to subtract $1$ from each element and our indices
are cyclic modulo $n$. These two posets are closely related to the posets of decorated
permutations and of Grassmann necklaces, considered in~\cite{Pos}.

We next consider the poset of cyclic rank matrices. These are infinite periodic matrices which relate to bounded affine permutations in the same way that Fulton's rank matrices relate to ordinary permutations. Finally, we will consider the poset of positroids. Introduced in~\cite{Pos}, these are matroids which obey certain positivity conditions.

The following is a combination of all the results of this section:

\begin{theorem} \label{thm:Bijections}
  The posets $\Q(k,n)$, $\Bound(k,n)$, $\Jugg(k,n)$, the poset of
  cylic rank matrices of type $(k,n)$ and the poset of positroids of
  rank $k$ on $[n]$ are all isomorphic.
\end{theorem}

The isomorphism between $\Bound(k,n)$ and cyclic rank matrices is Corollary~\ref{c:boundmanyfollowing}; the isomorphism between $\Jugg(k,n)$ and $\Bound(k,n)$ is Corollary~\ref{C:boundedjugg}; the isomorphism between $Q(k,n)$ and $\Bound(k,n)$ is Theorem~\ref{T:pairsboundedposet}; the isomorphism between $\Jugg(k,n)$ and positroids is Proposition~\ref{prop:JuggMatroid}.

\subsection{Juggling states and functions}\label{ssec:juggling}

\newcommand\shift{s_{+}} 
\newcommand\st{{\rm st}}

Define a \defn{(virtual) juggling state $S\subseteq \integers$} as a subset
whose symmetric difference from $-\naturals := \{ i \leq 0\}$
is finite. (We will motivate this and other juggling terminology below.)
Let its \defn{ball number} be
$\left|  S \cap \integers_+  \right | - \left| -\naturals \setminus S \right |$,
where $ \integers_+  := \{i > 0\}$.
Ball number is the unique function on juggling states
such that for $S \supseteq S'$,
the difference in ball numbers is $|S\setminus S'|$, and
$-\naturals$ has ball number zero.

Call a bijection $f:\Z \to \Z$ a \defn{(virtual) juggling function}
if for some (or equivalently, any) $t\in\Z$, the set $f\left(\{i : i
\leq t\}\right)$ is a juggling state. It is sufficient (but not
necessary) that $\{ |f(i)-i| : i \in \Z \}$ be bounded. Let $\JJ$ be
the set of such functions: it is easy to see that $\JJ$ is a group,
and contains the element $\shift: i \mapsto i+1$. Define the
\defn{ball number of $f\in \JJ$} as the ball number of the juggling
state $f(-\naturals)$, and denote it $\av(f)$ for reasons to be
explained later.

\begin{lem}
  $\av : \JJ \to \integers$ is a group homomorphism.
\end{lem}

\begin{proof}
  We prove what will be a more general statement,
  that if $S$ is a juggling state with ball number $b$,
  and $f$ a juggling function with ball number $b'$,
  then $f(S)$ is a juggling state with ball number $b+b'$.
  Proof: if we add one element to $S$, this adds one element to $f(S)$,
  and changes the ball numbers of $S,f(S)$ by $1$.
  We can use this operation and its inverse to reduce to the case that
  $S = -\naturals$, at which point the statement is tautological.

  Now let $f,g \in \JJ$, and apply the just-proven statement
  to $S = g(-\naturals)$.
\end{proof}

For any bijection $f:\Z \to Z$, let
\begin{eqnarray*}
   \st(f,t) &:=& \{f(i) - t: i \leq t\} \\
   &=& \st(\shift^t f \shift^{-t}, 0)
\end{eqnarray*}
and if $f\in \JJ$, call it the \defn{juggling state of $f$ at time
$t$}. By the homomorphism property just proven $\av(f) =
\av\left(\shift^k f \shift^{-k}\right)$, which says that every state
of $f\in \JJ$ has the same ball number (``ball number is
conserved''). The following lemma lets one work with juggling states
rather than juggling functions:

\begin{lem}\label{lem:reconstructjug}
  Say that a juggling state $T$ \defn{can follow} a state $S$ if
  $T = \{t\} \cup \left(\shift^{-1}\cdot S\right)$,
  and $t \notin \shift^{-1}\cdot S$. In this case say that a
  \defn{$t$-throw takes state $S$ to state $T$}.

  Then a list $(S_i)_{i\in\integers}$ is the list of states of a
  juggling function iff $S_{i+1}$ can follow $S_i$ for each $i$.
  In this case the juggling function is unique.
\end{lem}

\begin{proof}
  If the $(S_i)$ arise from a juggling function $f$, then the
  condition is satisfied where the element $t_i$ added
  to $\shift^{-1}\cdot S_{i-1}$ is $f(i)-i$.
  Conversely, one can construct $f$ as $f(i) = i+t_i$.
\end{proof}

In fact the finiteness conditions on juggling states and permutations
were not necessary for the lemma just proven. We now specify a further
finiteness condition, that will bring us closer to the true functions
of interest.

\begin{lem}\label{lem:finitestates}
  The following two conditions on a bijection $f:\Z\to\Z$ are equivalent:
  \begin{enumerate}
  \item there is a uniform bound on $|f(i)-i|$, or
  \item there are only finitely many different $\st(f,i)$
    visited by $f$.
  \end{enumerate}
  If they hold, $f$ is a juggling function.
\end{lem}

\begin{proof}
  Assume first that $f$ has only finitely many different states.
  By Lemma \ref{lem:reconstructjug}, we can reconstruct the value of $f(i)-i$
  from the states $S_i, S_{i+1}$. So $f(i)-i$ takes on only finitely
  many values, and hence $|f(i)-i|$ is uniformly bounded.

  For the reverse, assume that $|f(i)-i| < N$ for all $i\in\Z$.
  Then $f(-\naturals) \subseteq \{i < N\}$,
  and $f(\integers_+) \subseteq \{i >-N\}$.
  Since $f$ is bijective, we can complement the latter to learn that
  $f(-\naturals) \supseteq \{i\leq -N\}$. So $\st(f,0)$, and similarly
  each $\st(f,t)$, is trapped between $\{i\leq -N\}$ and $\{i < N\}$.
  There are then only $2^{2N}$ possibilities, all of which
  are juggling states.
\end{proof}

In the next section we will consider juggling functions which cycle
periodically through a finite set of states.

\newcommand\hght{{\rm ht}}
Define the \defn{height of the juggling state $S \subseteq \Z$} as
$$ \hght(S)
:= \sum_{i \in S \cap \integers_+} i - \sum_{i \in -\naturals \setminus S} i, $$
a sort of weighted ball number. We can now motivate the notation $\av(f)$,
computing ball number as an average:

\begin{lem}\label{lem:avheight}
  Let $a,b\in \integers, a\leq b$ and let $f\in \JJ$. Then
  $$ \sum_{i=a+1}^{b} (f(i)-i)
  = (b-a) \av(f) + \hght(\st(f,b)) - \hght(\st(f,a)).$$
  In particular, if $f$ satisfies the conditions of
  Lemma \ref{lem:finitestates}, then for any $a\in \integers$,
  $$ \lim_{b\to \infty} \frac{1}{b-a} \sum_{i=a+1}^{b} (f(i)-i)= \av(f). $$
  This equality also holds without taking the limit, if
  $\st(f,a) = \st(f,b)$.
\end{lem}

\begin{proof}
  It is enough to prove the first statement for $b=a+1$, and add the
  $b-a$ many equations together. They are of the form
  $$ f(a+1)-(a+1) = \av(f) + \hght(\st(f,a+1)) - \hght(\st(f,a)). $$
  To see this, start with $S = \st(f,a)$, and use $f(a+1)$ to calculate
  $\st(f,a+1)$. The three sets to consider are
  \begin{eqnarray*}
    S &=& f(\{i \leq a\}) \text{ shifted left by $a$} \\
    S'&=& f(\{i \leq a\}) \text{ shifted left by $a+1$} \\
    \st(f,a+1) &=& f(\{i \leq a+1\}) \text{ shifted left by $a+1$}
  \end{eqnarray*}
  By its definition, $\hght(S') = \hght(S) - \av(f)$.
  And $\hght(\st(f,a+1)) = \hght(S') + f(a+1)-(a+1)$. The equation follows.

  For the second, if $f$ only visits finitely many states then the
  difference in heights is bounded, and dividing by $b-a$ kills this
  term in the limit.
\end{proof}



We now motivate these definitions from a juggler's point of view. The canonical reference is \cite{Polster},
though our setting above is more general than considered there.
All of these concepts originated in the juggling community
in the years 1985-1990, though precise dates are difficult to determine.

Consider a idealized juggler who is juggling with one hand\footnote{%
  or as is more often assumed, rigidly alternating hands},
making one throw every second, of exactly one ball at a time,
has been doing so since the beginning of time and will continue
until its end. If our juggler is only human (other than being immortal) 
then there will be a limit on how high the throws may go.

Assume at first that the hand is never found empty when a throw
is to be made.
The history of the juggler can then be recorded by a function
$$
f(t) = \text{the time that a ball thrown at time $t$ is next thrown.}
$$
The number $f(t)-t$ is usually called the \defn{throw at time $t$}.
If ever the juggler {\em does} find the hand empty i.e. all the balls
in the air, then of course the juggler must wait one second for the
balls to come down. This is easily incorporated by taking $f(t)=t$,
a \defn{$0$-throw}.

While these assumptions imply that $f$ is a juggling function,
they would also seem to force the conclusion that $f(i)\geq i$,
i.e. that {\em balls land after they are thrown}. Assuming that for a
moment, it is easy to compute the number of balls being juggled in
the permutation $f$: at any time $t$, count how many balls were thrown
at times $\{i\leq t\}$ that are still in the air, $f(i)>t$.
This is of course our formula for the ball number, in this special case.
The formula $\av(f) = \av(\shift^t f \shift^{-t})$ then says that
balls are neither created nor destroyed.

The state of $f\in \JJ$ at time $t$ is the set of times in the
future (of $t$) that balls in the air are scheduled to land. (This
was introduced to study juggling by the first author and,
independently, by Jack Boyce, in 1988.) The ``height'' of a state
does not seem to have been considered before.

%

Thus, the sub-semigroup of $\JJ$ where $f(t) \geq t$ encodes
possible juggling patterns.  Since we would like to consider $\JJ$
as a group (an approach pioneered in~\cite{ER}), we must permit
$f(t) < t$. While it may seem fanciful to view this as describing
juggling with antimatter, the ``Dirac sea'' interpretation of antimatter is suggestive of the connection with the affine Grassmannian. 

\subsection{Affine permutations}\label{S:affineperm}

Let $\tS_n$ denote the group of bijections, called \defn{affine
  permutations}, $f: \Z \to \Z$ satisfying
$$
f(i + n) = i + n \ \ \mbox{for all $i \in \Z$}.
$$
Plainly this is a subgroup of $\JJ$. This group fits into an exact
sequence
$$ 1 \to \Z^n \xrightarrow{t} \tS_n \onto S_n \to 1 $$
where for $\mu = (\mu_1,\ldots,\mu_n) \in
\Z^n$, we define the \defn{translation element} $t_\mu \in \tS_n$ by
$t_\mu(i) = n\mu_i + i$ for $1 \leq i \leq n$.
The map $\tS_n \onto S_n$ is evident. We can give a splitting map $S_n \to \tS_n$ by
extending a permutation $\pi : [n] \to [n]$ periodically.
By this splitting, we have $\tS_n \simeq S_n \ltimes \Z^n$,
so every $f \in \tS_n$ can be uniquely factorized as $f =
w\ t_\mu$ with $w \in S_n$ and $\mu \in \Z^n$.

An affine permutation $f \in \tS_n$ is written in one-line notation
as $[\cdots f(1)f(2)\cdots f(n) \cdots]$ (or occasionally just as
$[f(1)f(2) \cdots f(n)]$). As explained in Section
\ref{ssec:jugafftnn}, jugglers instead list one period of the
periodic function $f(i)-i$ (without commas, because very few people
can make $10$-throws\footnote{%
  The few that do sometimes use $A,B,\ldots$ to denote throws $10,11,\ldots$,
  which prompts the question of what words are jugglable. Michael Kleber
  informs us that THEOREM and TEAKETTLE give valid juggling patterns.}
 and higher), and call this the
\defn{siteswap}. We adopt the same conventions when multiplying
affine permutations as for usual permutations. The ball number $
\av(f) = \frac{1}{n} \sum_{i=1}^n (f(i) - i) $ is
always an integer; indeed $\av(w\ t_\mu) = \av(t_\mu) = \sum_i
\mu_i$. Define
$$
\tS_n^k = \{f \in \tS_n \mid \av(f) = k\} = \shift^k \tS_n^0
$$
so that $\tS_n^0 = \ker \av$ is the Coxeter group with simple generators
$s_0,s_1,\ldots,s_{n-1}$, usually called the \defn{affine symmetric
  group}.\footnote{%
  One reason
  the subgroup $\tS_n^0$ is more commonly studied than $\tS_n$
  is that it is the Coxeter group $\widetilde{A_{n-1}}$;
  its relevance for us, is that it indexes the
  Bruhat cells on the affine flag manifold for the group $SL_n$.
  In \S \ref{sec:cohomology} we will be concerned with the affine flag manifold
  for the group $GL_n$, whose Bruhat cells are indexed by all of $\tS_n$.
}
Note that if $f \in \tS_n^a$ and $g \in \tS_n^b$ then the
product $fg$ is in $\tS_n^{a+b}$.  There is a canonical bijection
$f \mapsto f \circ (i \mapsto i+b-a)$
between the cosets $\tS_n^a$ and $\tS_n^b$.  The group $\tS_n^0$ has a Bruhat
order ``$\leq$'' because it is a Coxeter group $\widetilde{A_{n-1}}$.
This induces a partial order on each $\tS_n^a$, also denoted $\leq$.

An \defn{inversion} of $f$ is a pair $(i,j) \in \Z \times \Z$ such that $i
< j$ and $f(i) > f(j)$.  Two inversions $(i,j)$ and $(i',j')$ are
equivalent if $i' = i + rn$ and $j' = j + rn$ for some integer $r$.
The number of equivalence classes of inversions is the \defn{length}
$\ell(f)$ of $f$. This is sort of an ``excitation number'' of the
juggling pattern; this concept does not seem to have been studied
in the juggling community (though see \cite{ER}).

An affine permutation $f \in \tS_n^k$ is \defn{bounded} if $i \leq
f(i) \leq i+n$ for $i \in \Z$.  We denote the set of bounded affine
permutations by $\Bound(k,n)$.  The restriction of the Bruhat order
to $\Bound(k,n)$ is again denoted $\leq$.

\begin{lem}\label{L:orderideal}
The subset $\Bound(k,n) \subset \tS_n^k$ is a lower order ideal in
$(\tS_n^k,\leq)$.  In particular, $(\Bound(k,n),\leq)$ is graded by
the rank function $\ell(f)$.
\end{lem}
\begin{proof}
Suppose $f \in \Bound(k,n)$ and $g \lessdot f$.  Then $g$ is
obtained from $f$ by swapping the values of $i + kn$ and $j + kn$
for each $k$, where $i < j$ and $f(i) > f(j)$.  By the assumption on the
boundedness of $f$, we have $i+n \geq f(i) > f(j) = g(i) \geq j > i$
and $j +n > i + n \geq f(i) = g(j) > f(j) \geq j$.  Thus $g \in
\Bound(k,n)$.
\end{proof}

Postnikov, in \cite[\S 13]{Pos}, introduces ``decorated permutations''.
A decorated permutation is an element of $S_n$, with each fixed point colored either $1$ or $-1$.
There is an obvious bijection between the set of decorated permutations and $\coprod_{k=0}^n \Bound(k,n)$:
Given an element $f \in \Bound(k,n)$, form the corresponding decorated permutation by reducing $f$ modulo $n$
and coloring the fixed points of this reduction $-1$ or $1$ according to whether $f(i)=i$ or $f(i)=i+n$ respectively.
In~\cite[\S 17]{Pos}, Postnikov introduces the cyclic Bruhat order, $CB_{kn}$, on those decorated permutations corresponding to elements of $\Bound(k,n)$.
From the list of cover relations in~\cite[Theorem 17.8]{Pos}, it is easy to see that $CB_{kn}$ is anti-isomorphic to $\Bound(k,n)$.

\begin{example}
  In the $\Gr(2,4)$ case there are already $33$ bounded affine permutations,
  but only $10$ up to cyclic rotation. In Figure \ref{fig:gr24} we show
  the posets of siteswaps, affine permuations, and decorated permutations,
  each modulo rotation. Note that the cyclic symmetry is most visible on the
  siteswaps, and indeed jugglers draw little distinction between
  cyclic rotations of the ``same'' siteswap.
  \begin{figure}
    \centering
    \epsfig{file=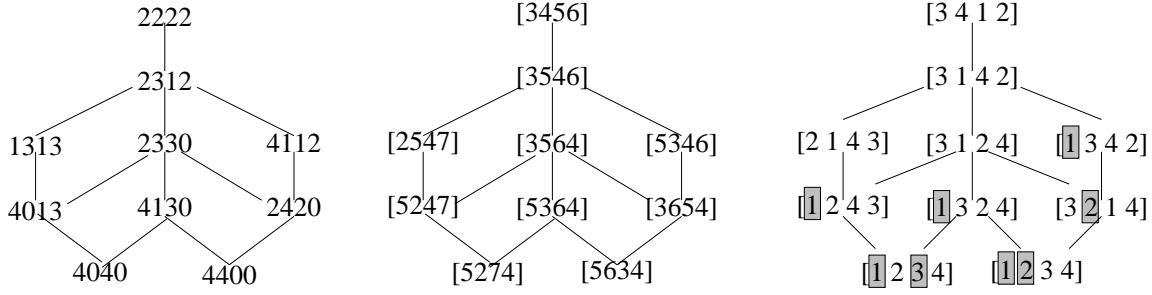,width=6in}
    \caption{The posets of siteswaps, bounded affine permutations,
      and decorated permutations for $\Gr(2,4)$, each up to
      cyclic rotation. (The actual posets each have 33 elements.)}
    \label{fig:gr24}
  \end{figure}
\end{example}

\subsection{Sequences of juggling states}
A \defn{$(k,n)$-sequence of juggling states} is a sequence $\J =
(J_1,\ldots, J_n) \in \binom{[n]}{k}^n$ such that for each $i \in
[n]$, we have that $J_{i+1} \cup -\naturals$ follows $J_i \cup -\naturals$,
where the indices are taken modulo $n$.  Let $\Jugg(k,n)$ denote the
set of such sequences.

Let $f \in \Bound(k,n)$.  Then the sequence of juggling states
$$
\ldots, \st(f,-1), \st(f,0), \st(f,1), \ldots
$$
is periodic with period $n$.  Furthermore for each $i \in \Z$, (a)
$-\naturals \subset \st(f,i)$, and (b) $\st(f,i) \cap [n] \in
\binom{[n]}{k}$.  Thus
$$
\J(f) = (\st(f,0) \cap [n], \st(f,1) \cap [n], \ldots, \st(f,n-1)
\cap [n]) \in \Jugg(k,n).
$$

\begin{lem}
The map $f \mapsto \J(f)$ is a bijection between $\Bound(k,n)$ and
$\Jugg(k,n)$.
\end{lem}

We now discuss another way of viewing $(k,n)$-sequences of juggling
states which will be useful in \S \ref{ssec:crank}.  Let $S$ be
a $k$-ball virtual juggling state.  For every integer $j$, define
$$R_j(S) = k - \# \{x \in S: \ x > j \}.$$
These $\{R_j\}$ satisfy the following properties:
\begin{itemize}
\item $R_j(S) - R_{j-1}(S)$ is either $0$ or $1$, according to whether
  $j \not \in J$ or $j \in J$ respectively,
\item $R_j(S) = k$ for $j$ sufficently positive, and
\item $R_j(S) = j$ for $j$ sufficiently negative.
\end{itemize}
Conversely, from such a sequence $(R_j)$ one can construct a
$k$-ball juggling state.


Let $S_1$ and $S_{2}$ be two $k$-ball juggling states. Define a $2
\times \infty$ matrix $(r_{ij})$ by $r_{ij} = R_{j-i+1}(S_i)$.

\begin{lemma} \label{l:following}
  The state $S_2$ can follow $S_1$ if and only if $r_{1j}-r_{2j} =0$
  or $1$ for all $j \in \Z$ and there is no $2 \times 2$ submatrix for
  which $r_{1j}=r_{2j}=r_{2(j+1)}=r_{1(j+1)}-1$.
\end{lemma}

\begin{proof}
  It is easy to check that $S_2 = \shift^{-1} S_1 \cup \{ t \}$ if and
  only if $r_{1j} = r_{2j}+1$ for $j \leq t$ and $r_{1j} = r_{2j}$ for
  $j \geq t+1$.  If this holds, it immediately follows that
  $r_{1j}-r_{2j} =0$ or $1$ for all $j$ and that there is no $j$ for
  which $r_{1j}=r_{2j}$ while $r_{2(j+1)}=r_{1(j+1)}-1$.

  Conversely, suppose that $r_{1j}-r_{2j} =0$ or $1$ for all $j \in
  \Z$ and there is no $j$ for which
  $r_{1j}=r_{2j}=r_{2(j+1)}=r_{1(j+1)}-1$.  Then we claim that there
  is no $j$ for which $r_{1j}=r_{2j}$ and $r_{2(j+1)}=r_{1(j+1)}-1$.
  Proof: suppose there were. If $r_{2(j+1)}=r_{2j}$, then we are done
  by our hypothesis; if $r_{2(j+1)} = r_{2 j}$ then $r_{1(j+1)} =
  r_{1j}+2$, contradicting that $r_{1(j+1)}-r_{1j}=0$ or $1$.  Since
  $r_{2(j+1)}-r_{2j}=0$ or $1$, we have a contradiction either
  way. This establishes the claim.  Now, we know that $r_{1j}=r_{2j}$
  for $j$ sufficently positive and $r_{1j}=r_{2j}+1$ for $j$
  sufficently negative, so there must be some $t$ such that $r_{1j} =
  r_{2j}+1$ for $j \leq t$ and $r_{1j} = r_{2j}$ for $j \geq t+1$.
  Then $S_{2}$ can follow $S_1$.
\end{proof}

It is immediate to extend this result to a sequence of juggling
states.  Let $\JJ$ be the group of juggling functions introduced in
\S \ref{ssec:juggling} and let $\JJ^{\av=k}$ be those juggling
functions with ball number $k$.  For any $f$ in $\JJ^{\av=k}$, let
$\J(f) = (J_1, J_2,\ldots,J_{n})$ be the corresponding
$(k,n)$-sequence of juggling states. Define an $\infty \times
\infty$ matrix by $r_{ij} = R_{(j-i+1)}(J_i \cup -\naturals)$. Then,
applying Lemma~\ref{l:following} to each pair of rows of $(r_{ij})$
gives:

\begin{cor} \label{c:manyfollowing} The above construction gives a
  bijection between $\JJ^{\av=k}$ and $\infty \times \infty$ matrices such
  that
  \begin{enumerate}
  \item[(C1)] for each $i$, there is an $m_i$ such that $r_{ij} =
    j-i+1$ for all $j \leq m_i$,
  \item[(C2)] for each $i$, there is an $n_i$ such that $r_{ij} = k$ for all $j \geq n_i$,
  \item[(C3)] $r_{ij} - r_{(i+1)j} \in \{0,1\}$ and $r_{ij} - r_{i(j-1)} \in
    \{0,1\}$ for all $i, j \in \Z$, and
  \item[(C4)] if $r_{(i+1)(j-1)}=r_{(i+1)j}=r_{i(j-1)}$ then $r_{ij}=r_{(i+1)(j-1)}$.
  \end{enumerate}

  Under this bijection, $r_{ij} = r_{(i+1)j} = r_{i(j-1)} >
  r_{(i+1)(j-1)}$ if and only if $f(i)=j$.
\end{cor}


\begin{prop} \label{prop:rankBruhat}
Let $f, g \in \tS_n^k \subset \JJ^{\av = k}$ and let $r$ and $s$ be
the corresponding matrices. Then $f \leq g$ (in Bruhat order) if and
only if $r_{ij} \geq s_{ij}$ for all $(i,j) \in \Z^2$.
\end{prop}

\begin{proof}
  See \cite[Theorem 8.3.1]{BjB}.
\end{proof}

When $r_{ij} = r_{(i+1)j} = r_{i(j-1)} > r_{(i+1)(j-1)}$, we say
that $(i,j)$ is a \defn{special entry} of $r$.

An easy check shows:
\begin{cor} \label{c:boundmanyfollowing}
Under the above bijection, $\Bound(k,n)$ corresponds to $\infty \times \infty$ matrices such that
\begin{enumerate}
\item[(C1')] $r_{ij} = j-i+1$ for all $j < i$,
\item[(C2')] $r_{ij} = k$ for all $j \geq i+n-1$,
\item[(C3)] $r_{ij} - r_{(i+1)j} \in \{0,1\}$ and $r_{ij} - r_{i(j-1)} \in
\{0,1\}$ for all $i, j \in \Z$,
\item[(C4)] if $r_{(i+1)(j-1)}=r_{(i+1)j}=r_{i(j-1)}$ then $r_{ij}=r_{(i+1)(j-1)}$, and
\item[(C5)] $r_{(i+n)(j+n)} = r_{ij}$.
\end{enumerate}
\end{cor}

We call a matrix $(r_{ij})$ as in Corollary~\ref{c:boundmanyfollowing} a \defn{cyclic rank matrix}.
(See~\cite{Fulton} for the definition of a rank matrix, which we are mimicking.)
We now specialize Proposition~\ref{prop:rankBruhat} to the case of $\Bound(k,n)$:
Define a partial order $\leq$ on $\Jugg(k,n)$ by
$$
(J_1,\ldots, J_n) \leq (J'_1,\ldots, J'_n) \ \ \mbox{if and only if
$J_i \leq J'_i$ for each $i$.}
$$

\begin{cor}\label{C:boundedjugg}
The map $f \mapsto \J(f)$ is an isomorphism of posets from bounded
affine permutations $(\Bound(k,n),\leq)$ to $(k,n)$-sequence of
juggling states $(\Jugg(k,n),\leq)$.
\end{cor}

\begin{proof}
 One simply checks that the condition $r_{ij} \geq r'_{ij}$ for all $j$ is equivalent to $J_i \leq J'_i$.
\end{proof}

\begin{example}
Let $n = 4$ and $k = 2$.  Consider the affine permutation $[\cdots
2358 \cdots]$, last seen in Figure \ref{fig:gr24}. Its siteswap is
$4112$, and the corresponding sequence of juggling states is $(14,
13, 12, 12)$. Below we list a section of the corresponding infinite
permutation matrix and cyclic rank matrix.  Namely, we display the
entries $(i,j)$ for $1 \leq i \leq 4$ and $i \leq j \leq i+4$.  The
special entries have been underlined.
$$\begin{pmatrix}
0 & 1 & 0 & 0 & 0 \\
  &   0 & 1 & 0 & 0 & 0 &        \\
  &   & 0 & 0 & 1 & 0 & 0     \\
  &   &   & 0 & 0 & 0 & 0 & 1
\end{pmatrix} \quad
\begin{pmatrix}
1 & \underline{1} & 1 & 2 & 2  \\
  & 1 & \underline{1} & 2 & 2 & 2 \\
  &   & 1 & 2 & \underline{2} & 2 & 2 \\
  &   &   & 1 & 2 & 2 & 2 & \underline{2}
\end{pmatrix}
$$

\end{example}

In \S \ref{ssec:jugafftnn} we associated an affine permutation to 
each $k\times n$ matrix $M$ of rank $k$; 
a modification of that rule gives instead a $(k,n)$-sequence of juggling states.
Call a column of $M$ \defn{pivotal} if it
is linearly independent of the columns to its left.
(If one performs Gaussian elimination on $M$, a column will
be pivotal exactly if it contains a ``pivot'' of
the resulting reduced row-echelon form.)
There will be $k$ pivotal columns,
giving a $k$-element subset of $\{1,\ldots,n\}$; they form the
lex-first basis made of columns from $M$.

Now rotate the first column of $M$ to the end.
What happens to the set of pivotal columns?
Any column after the first that was pivotal still is pivotal, but (unless the
first column was all zeroes) there is a new pivotal column; 
the new state can follow the previous state.
The $n$ cyclic rotations of $M$ thus give 
a $(k,n)$-sequence of juggling states.

\junk{
  Assume that the $k$ pivotal numbers in $\{1,\ldots,n\}$ describe the times
  that the balls currently in the air are going to land. If there is no ball in
  the hand right now (the first column being $\vec 0$), then all the juggler can
  do is wait $1$ second for the balls to come lower.
  Otherwise there is a ball in the hand now,
  which the juggler throws so that it will come down $t$ steps in the
  future; this is called a {\em $t$-throw}.
  (The empty hand case is called a $0$-throw.)

  Postnikov proves that
  the nearly-identical notion of ``Grassmann necklaces'' (to which these
  juggling patterns biject in a simple way) is a complete invariant
  of the strata, and all necklaces arise.
}

\subsection{From $\Q(k,n)$ to $\Bound(k,n)$}\label{S:pairstobounded}
The symmetric group $S_n$ acts on $\Z^n$ (on the left) by
\begin{equation}\label{E:Snaction}
w \cdot (\omega_1,\ldots,\omega_n) = (\omega_{w^{-1}(1)}, \ldots
\omega_{w^{-1}(n)}).
\end{equation}
If $w \in S_n$ and $t_\lambda, t_{\lambda'} \in \tS_n$ are translation
elements, 
we have the following relations in $\tS_n$:
\begin{equation}\label{E:translation} wt_\lambda w^{-1} = t_{w \cdot \lambda} \ \
\ \ \ \ t_{\lambda} t_{\lambda'} = t_{\lambda+ \lambda'}.
\end{equation}

Let $\omega_k = (1,\ldots,1,0,\ldots,0)$ with $k$ $1$s be the $k$th
fundamental weight of $GL(n)$.  Note that $t_{\omega_k} \in \tS_n^k$.  Now fix
$\langle u,w \rangle \in \Q(k,n)$, the set of equivalence classes
we defined in \S \ref{sec:kBruhat}.
Define an affine permutation $f_{u,w} \in \tS_n^k$ by
$$
f_{u,w} = ut_{\omega_k}w^{-1}.
$$
The element $f_{u,w}$ does not depend on the representative
$[u,w]_k$ of $\langle u,w \rangle $: if $u' = uz$ and $w' = wz$ for
$z \in S_k \times S_{n-k}$ then
$$
u't_{\omega_k}(w')^{-1} = uzt_{\omega_k}z^{-1}w^{-1} = ut_{z\cdot
\omega_k} w^{-1} = u t_{\omega_k} w^{-1}
$$
since $z$ stabilizes $\omega_k$.

\begin{prop}\label{P:pairstobounded}
The map $\langle u,w \rangle \mapsto f_{u,w}$ is a bijection from
$\Q(k,n)$ to $\Bound(k,n)$.
\end{prop}
\begin{proof}
We first show that $\langle u,w \rangle \mapsto f_{u,w}$ is an
injection into $\tS_n^k$.
Suppose that $f_{u,w} = f_{u', w'}$.
It is clear from the factorization
$\tS_n \simeq S_n \ltimes \Z^n$ that there is some $z \in S_k \times S_{n-k}$ such that $u=u' z$ and $w = w' z$. 
Using Proposition~\ref{prop:Equivalence}, we have $[u,w] \sim [u', w']$.

We now show that for $\langle u,w \rangle \in
\Q(k,n)$, we have $f_{u,w} \in \Bound(k,n)$. Let $i \in [1,n]$ and
$a = w^{-1}(i)$. Then
$$
f_{u,w}(i) = \begin{cases} u(a) & \mbox{if $a > k$}
\\
u(a) + n & \mbox{if $a \leq k$.}
\end{cases}
$$
The boundedness of $f_{u,w}$ now follows from Theorem
\ref{T:BScriterion}(1).

Conversely, if $f \in \Bound(k,n)$ then it is clear from
\eqref{E:translation} that $f$ has a factorization as
$$
f = u t_\omega w^{-1}
$$
for $u, w \in S_n$ and $\omega \in \{0,1\}^n$.  Since $f \in
\tS_n^k$, the vector $\omega$ has $k$ $1$s.  By changing $u$ and $w$,
we may further assume that $\omega = \omega_k$ and $w \in \GS$.
It remains to check that $u \leq_k w$, which we do via Theorem
\ref{T:BScriterion}; its  Condition (2) is vacuous when $w \in \GS$ and
checking Condition (1) is the same calculation as in the previous
paragraph.
\end{proof}

\begin{thm}\label{T:pairsboundedposet}
The bijection $\langle u,w \rangle \mapsto f_{u,w}$ is a poset isomorphism from
the pairs $(\Q(k,n),\leq)$ to bounded affine permutations
$(\Bound(k,n),\leq)$. Furthermore, one has $\ell(f_{u,w}) =
\binom{n}{k} - \ell(w) + \ell(u)$.
\end{thm}
\begin{proof}
It is shown in \cite{Wil} that $(\Q(k,n), \leq)$ is a graded poset,
with rank function given by $\rho(\langle u,w \rangle ) = k(n-k) -
(\ell(w) - \ell(u))$. It follows that each cover in $\Q(k,n)$ is of
the form
\begin{enumerate}
\item $\langle u,w' \rangle \gtrdot \langle u,w \rangle$ where $w' \lessdot w$, or
\item $\langle u',w \rangle  \gtrdot \langle u,w \rangle$ where $u \lessdot u'$.
\end{enumerate}
We may assume that
$w \in \GS$ (Proposition~\ref{prop:GrassRep}).  Suppose we are in Case (1).  Then $w' = w(ab)$ where
$a \leq k < b$ and $w(a) > w(b)$.  Here $(ab) \in S_n$ denotes the
transposition swapping $a$ and $b$.  Thus $f_{u,w'} =
f_{u,w}(w(a)w(b))$. Using the formula in the proof of Proposition
\ref{P:pairstobounded}, we see that $f_{u,w}(w(a)) > n$ while
$f_{u,w}(w(b)) \leq n$.  Thus $f_{u,w'} > f_{u,w}$.

Suppose we are in Case (2), and that $u' = u(ab)$ where $a < b$ and
$u(a) < u(b)$.  It follows that $f_{u,w'} = (u(a)u(b))f_{u,w}$.
Suppose first that $a \leq k < b$.  Then $(t_{\omega_k})^{-1}(a) = a
- n$, while $(t_{\omega})^{-1}(b) = b$ so we also have $f_{u,w'} =
f_{u,w}((w(a)-n)w(b))$ where $w(a) - n$ is clearly less than
$w(b)$.  Thus $f_{u,w'} > f_{u,w}$.  Otherwise suppose that $a,
b \geq k$ (the case $a,b \leq k$ is similar).  Then $f_{u,w'} =
f_{u,w}(w(a)w(b))$.  Since $w \in \GS$, we have $w(a) < w(b)$.
Again we have $f_{u,w'} > f_{u,w}$.

We have shown that $\langle u',w'\rangle \geq \langle u,w\rangle$
implies $f_{u',w'} \geq f_{u,w}$.  The converse direction is
similar.

The last statement follows easily, using the fact that both of the
posets $(\Q(k,n),\leq)$ and $(\Bound(k,n),\leq)$ are graded.
\end{proof}


\subsection{Shellability of $\Q(k,n)$}\label{ssec:shellability}

A graded poset $P$ is \defn{Eulerian} if for any $x \leq y \in P$
such that the interval $[x,y]$ is finite we have $\mu(x,y) = (-1)^{{\rm
rank}(x) - {\rm rank}(y)}$, where $\mu$ denotes the M\"{o}bius
function of $P$.  A labeling of the Hasse diagram of a poset $P$ by
some totally ordered set $\Lambda$ is called an \defn{$EL$-labeling}
if for any $x \leq y \in P$:
\begin{enumerate}
\item
  there is a unique label-(strictly)increasing saturated chain $C$
  from $x$ to $y$,
\item
  the sequence of labels in $C$ is $\Lambda$-lexicographically minimal amongst
  the labels of saturated chains from $x$ to $y$.
\end{enumerate}
If $P$ has an $EL$-labeling then we say that $P$ is {\it
$EL$-shellable}.

Verma \cite{Ver} has shown that the Bruhat order of a Coxeter group
is Eulerian.  Dyer \cite[Proposition 4.3]{Dye} showed the stronger
result that every Bruhat order (and also its dual) is
$EL$-shellable.  (See also~\cite{BW}.) Since these properties are
preserved under taking convex subsets, Lemma \ref{L:orderideal} and
Corollary \ref{C:boundedjugg} and Theorem \ref{T:pairsboundedposet}
imply the following result, proved for the dual of $(\Q(k,n), \leq)$
by Williams \cite{Wil}.

\begin{cor}\label{cor:shell}
  The posets $(\Q(k,n), \leq)$, $(\Bound(k,n),\leq)$, and
  $(\Jugg(k,n),\leq)$, and their duals are Eulerian and
  $EL$-shellable.
\end{cor}

\begin{remark}
  Williams' result is stronger than
  Corollary~\ref{cor:shell}: in our language, she shows that the poset
  $\widehat{\Q(k,n)}$, formed by adding a formal maximal element
  $\hat{1}$ to $\Q(k,n)$, is shellable.
\end{remark}

\subsection{Positroids}\label{S:positroids}
A \defn{matroid $\M$ on $[n]$ with rank $k$} is a non-empty
collection of $k$-element subsets of $[n]$, called \defn{bases},
satisfying the Unique Minimum Axiom: For any permutation $w \in
S_n$, there is a unique minimal element of $w \cdot \M$, in the
partial order $\leq$ on $\binom{[n]}{k}$.  This is only one of many
equivalent definitions of a matroid; see~\cite{Bry} for a compendium
of many others, in which this one appears as Axiom $B2^{(6)}$.

Let $\M$ be a matroid of rank $k$ on $[n]$.  Define a sequence of
$k$-element subsets $\J(\M) = (J_1, J_2, \ldots, J_n)$ by letting
$J_r$ be the minimal base of $\chi^{-r+1}(\M)$, which is
well-defined by assumption. Postnikov proved, in the terminology of
Grassmann necklaces,

\begin{lem}[{\cite[Lemma 16.3]{Pos}}]
For a matroid $\M$, the sequence $\J(\M)$ is a $(k,n)$-sequence of
juggling states.
\end{lem}

Let $\J =(J_1,J_2,\ldots,J_r) \in \Jugg(k,n)$.  Define
$$
\M_\J = \left\{I \in \binom{[n]}{k} \mid \chi^{-r+1}(I) \geq J_r\right\}.
$$

\begin{lem}[\cite{Pos, Oh}]
  Let $\J \in \Jugg(k,n)$.  Then $\M_\J$ is a matroid and $\J(\M_{\J}) = \J$.
\end{lem}
The matroids $\M_\J$ are called \defn{positroids}. 

\begin{prop} \label{prop:JuggMatroid}
The maps $\J \mapsto \M_{\J}$ and $\M \mapsto \J(\M)$ are inverse isomorphisms between the poset $\Jugg(k,n)$ and the poset of positroids, $\J(\M_1) \leq \J(\M_2)$ if and only if $\M_1 \supseteq \M_2$. 
\end{prop}

\begin{proof}
The composition $\J \mapsto \M_{\J} \mapsto \J(\M_{\J})$ is the identity by the above lemma and, since the set of positroids is defined as those matroids of the form $\M_{\J}$, the compositions are inverse in the other order as well.

It is easy to see from the definitions, that $M_1 \supseteq M_2$ implies $\J(M_1) \leq \J(M_2)$ and that $\J_1 \leq \J_2$ implies $\M_{\J_1} \supseteq \M_{\J_2}$. Since these correspondences are inverse, then $\J_1 \leq \J_2$ if and only if $\M_{\J_1} \supseteq \M_{\J_2}$.
\end{proof}

If $\M$ is an arbitrary
matroid, then we call the positroid $\M_{\J(\M)}$ the
\defn{positroid envelope} of $\M$ (see the discussion
before Remark \ref{rem:SashaParam}).  Every positroid is a matroid.
The positroid envelope of a positroid is itself.

\begin{example}
  Let $M_1$ and $M_2$ be the matroids $\{ 12, 13, 14, 23,
  24, 34 \}$ and $\{ 12, 23, 34, 14 \}$.  In both cases, $\J({M_i})$ is
  $( 12, 23, 34, 14 )$ and, thus, $M_1$ is the positroid envelope of
  both $M_1$ and $M_2$.  The corresponding affine permutation is
  $[\cdots 3456 \cdots]$.  On the other hand, if $M_3 = \{ 12, 13, 14, 23, 24 \}$,
  then $\J(M_3) = \{ 12, 23, 13, 14 \}$, with corresponding affine
  permutation $[\cdots 3546 \cdots]$.
\end{example}

\begin{remark}
Postnikov \cite{Pos} studied the totally nonnegative part
$\Gr(k,n)_{\geq 0}$ of the Grassmannian.  Each point $V \in
\Gr(k,n)_{\geq 0}$ has an associated matroid $\M_V$.  Postnikov
showed that the matroids that can occur this way, called positroids,
were in bijection with Grassmann necklaces of type $(k,n)$ (similar
to our $(k,n)$-sequences of juggling states), with decorated
permutations of $[n]$ with $k$ anti-exceedances, and with many other
combinatorial objects. Oh \cite{Oh}, proving a conjecture of
Postnikov, showed that positroids can be defined in the way we have
done.
\end{remark}

\section{Background on Schubert and Richardson varieties}

\newcommand\Project{\mathrm{Project}}

We continue to fix nonnegative integers $k$ and $n$, satisfying $0
\leq k \leq n$. For $S$ any subset of $[n]$,
let $\Project_S : \complexes^n \onto \complexes^S$
denote the projection onto the coordinates indexed by $S$.
(So the kernel of $\Project_S$ is $\Span_{s \not \in S} e_s$.)

\subsection{Schubert and Richardson varieties in the flag manifold}
\label{ssec:SchubertRichardsonInFlag}

Let $\Fl(n)$ denote the variety of flags in $\C^n$.
For a permutation $w \in S_n$, we have the
\defn{Schubert cell}
$$
\Xo_w = \{G_\bullet \in \Fl(n) \mid \dim(\Project_{[j]}(G_i)) = \#\{w([i]) \cap [j] \} \; {\rm for \; all}\; i,j\}
$$
and \defn{Schubert variety}
$$
X_w = \{G_\bullet \in \Fl(n) \mid \dim(\Project_{[j]}(G_i)) \leq \#\{w([i]) \cap [j] \} \; {\rm for \; all}\; i,j\}
$$
which both have codimension $\ell(w)$;
moreover $X_w = \overline{\Xo_w}$.
(For basic background on the combinatorics of Schubert varieties, see~\cite{Fulton} or~\cite[Chapter 15]{MS}.)
We thus have
$$\Fl(n) =
\coprod_{w \in S_n} \Xo_w \ \ \ \ \ {\rm and} \ \ \ \ \ X_w =
\coprod_{v \geq w} \Xo_v.
$$

Similarly, we define the
\defn{opposite Schubert cell}
$$
\Xo^{w} = \{G_\bullet \in \Fl(n) \mid \dim(\Project_{[n-j+1,n]}(G_i)) = \#\{w([i]) \cap [n-j+1,n] \} \; {\rm for \; all}\; i,j\}
$$
and \defn{opposite Schubert variety}
$$
X^{w} = \{G_\bullet \in \Fl(n) \mid \dim(\Project_{[n-j+1,n]}(G_i)) \leq \#\{w([i]) \cap [n-j+1,n] \} \; {\rm for \; all}\; i,j\}
$$

It may be easier to understand these definitions in terms of matrices.
Let $M$ be an $n \times n$ invertible matrix and let $G_i$ be the span of the top $i$ rows of $M$.
Then $G_{\bullet}$ is in $\Xo_w$ (respectively, $X_w$), if and only if,
for all $1 \leq i,j \leq n$, the rank of the top left $i \times j$ submatrix of $M$ is the same as
(respectively, less than or equal to) the rank of the corresponding submatrix of the permutation matrix $w$.
Similarly, $G_{\bullet}$ is in $\Xo^{w}$ (respectively $X^w$) if the ranks of the top right
submatrices of $M$ are equal to (respectively less than or equal to) those of $w$.
(The permutation matrix of $w$ has $1$s in positions $(i, w(i))$ and $0$s elsewhere.)

Define the \defn{Richardson varieties} as the transverse intersections
$$
X_u^w = X_u \cap X^w \ \ \ \ \ {\rm and} \ \ \ \ \
\Xo_{u}^{w} = \Xo_u \cap \Xo_w.
$$
The varieties $X_v^w$ and $\Xo_{v}^{w}$ are nonempty if and
only if $v \leq w$, in which case each has dimension $\ell(w) -
\ell(v)$. Let $E_{\bullet}$ be the flag $(\Span(e_1), \Span(e_1, e_2), \ldots)$
The coordinate flag $v E_{\bullet}$ is in $X_u^w$ if
and only if $u \leq v \leq w$.

We will occasionally need to define \defn{Schubert cells and varieties with respect to a flag $F_{\bullet}$}.
We set
$$
\Xo_w(F_\bullet) = \{G_\bullet \in \Fl(n) \mid \dim(G_i/(G_i \cap
F_{n-j})) = \#\{w([i]) \cap [j] \} \; {\rm for \; all}\; i,j\}
$$
and define $X_w(F_{\bullet})$ by replacing $=$ with $\leq$.
Warning: under this definition $X_w$ is $X_w(w_0 E_{\bullet})$.

\subsection{Schubert varieties in the Grassmannian}
Let $\Gr(k,n)$ denote the Grassmannian of $k$-planes in $\C^n$, and
let $\pi: \Fl(n) \to \Gr(k,n)$ denote the natural projection.  For
$I \in \binom{[n]}{k}$, we let
$$
\Xo_I = \{V \in \Gr(k,n) \mid \dim \Project_{[j]}(V)) = \#(I \cap [j]) \}
$$
denote the Schubert cell labeled by $I$ and
$$
X_I = \{V \in \Gr(k,n) \mid \dim \Project_{[j]}(V) \leq \#(I \cap [j]) \}
$$
the Schubert variety labeled by $I$.

Thus we have $\pi(X_w) =
X_{\pi(w)}$ and
$$\Gr(k,n) =
\coprod_{I \in \binom{[n]}{k}} \Xo_I \ \ \ \ \ {\rm
and} \ \ \ \ \ X_J = \coprod_{I \geq J} \Xo_I.
$$

We define
$$
\Xo^{I} = \{V \in \Gr(k,n) \mid \dim \Project_{[n-j+1,n]}(V) = \#(I \cap [n-j+1,n]) \},
$$
$$
X^I = \{V \in \Gr(k,n) \mid \dim \Project_{[n-j+1,n]}(V) \leq \#(I \cap [n-j+1,n]) \},
$$

So, for $J \in \binom{[n]}{k}$, the $k$-plane $\Span_{j \in J} e_j$ lies in $X_I$ if and only if $I \leq J$, and lies in $X^K$ if and only if $J \leq K$.

To review: if $u$ and $w$ lie in $S_n$, then $X_u$ is a Schubert variety, $X^w$ an opposite Schubert and $X_u^w$ a Richardson variety in $\Fl(n)$.
If $I$ and $J$ lie in $\binom{[n]}{k}$, then $X_I$, $X^J$ and $X_I^J$ mean the similarly named objects in $\Gr(k,n)$.
(Note that permutations have lower case letters from the end of the alphabet while subsets have upper case letters chosen from the range $\{ I, J, K \}$.)
The symbol $\Xo$ would indicate that we are dealing with an open subvariety, in any of these cases.

\section{Positroid varieties} \label{sec:pos}

We now introduce the positroid varieties, our principal objects of
study. Like the Schubert and Richardson varieties, they will come in
open versions, denoted $\Pio$, and closed versions, denoted
$\Pi$.\footnote{%
$\Pi$ stands for ``positroid", ``Postnikov", and ``projected Richardson''.}
The positroid varieties will be subvarieties of $\Gr(k,n)$, indexed by
the various posets introduced in \S \ref{sec:posets}. For each of
the different ways of viewing our posets, there is a corresponding way
to view positroid varieties. The main result of this section will be
that all of these ways coincide. Again, we sketch these results here
and leave the precise definitions until later.

Given $[u,w]_{k}$, representing an equivalence class in $\Q(k,n)$,
we can project the Richardson variety $\Xo_u^w$ (respectively
$X_u^w$) to $\Gr(k,n)$. Given a $(k,n)$-sequence of juggling states
$(J_1, J_2, \ldots, J_n) \in \Jugg(k,n)$, we can take the
intersection $\bigcap \chi^{i-1} \Xo_{J_i}$ (respectively $\bigcap
\chi^{i-1} X_{J_i}$) in $\Gr(k,n)$. (Recall $\chi$ is the cyclic
shift $[234\ldots n1]$.) Given a cyclic rank matrix $r$, we can
consider the image in $\Gr(k,n)$ of the space of $k \times n$
matrices such that the submatrices made of cyclically consecutive
columns have ranks equal to (respectively, less than or equal to)
the entries of $r$. Given a positroid $M$, we can consider those
points in $\Gr(k,n)$ whose matroid has positroid envelope equal to
(respectively, contained in) $M$.

\begin{theorem}
  Choose our $[u,w]_{k}$, $(J_1, \ldots, J_n)$, $r$ and $M$ to correspond
  by the bijections in \S \ref{sec:posets}. Then the projected open
  Richardson variety, the intersection of cyclically permuted open
  Schubert varieties, the
  space of matrices obeying the rank conditions, and the space of
  matrices whose matroids have the required positroid envelope, will
  all coincide as subsets of $\Gr(k,n)$.
\end{theorem}

The equalities of the last three spaces is essentially a matter of unwinding definitions.
The equality between the projected open Richardson variety, and the space of matrices obeying the cyclic rank conditions, is nontrivial and is checked in Proposition~\ref{P:triplecyclicrank}.

We call the varieties we obtain in this way \defn{open positroid
  varieties} or \defn{positroid varieties} respectively, and denote
them by $\Pio$ or $\Pi$ with a subscript corresponding to any of the
possible combinatorial indexing sets.

The astute reader will note that we did not describe how to define a
positroid variety using a bounded affine permutation (except by
translating it into some other combinatorial data).  
We hope to address this in future work using the geometry of the affine flag manifold. The significance of bounded
affine permutations can already be seen in this paper, as it is
central in our description in \S \ref{sec:cohomology}
of the cohomology class of $\Pi$.

\subsection{Cyclic rank matrices} \label{ssec:crank}
Recall the definition of a cyclic rank matrix from the end of
\S \ref{ssec:juggling}.  As we explained there, cyclic rank
matrices of type $(k,n)$ are in bijection with $\Bound(k,n)$ and hence
with $\Q(k,n)$ and with bounded juggling patterns of type $(k,n)$.

Let $V \in \Gr(k,n)$.  We define an infinite array $r_{\bullet\bullet}(V) = (r_{ij}(V))_{i, j \in \Z}$ of integers as follows: For $i>j$, we set $r_{ij}(V)=j-i+1$ and for $i \leq j$ we have
$$
r_{ij}(V) = \dim (\Project_{\{i, i+1, \ldots, j \}}(V) \}).
$$
where the indices are cyclic modulo $n$. (So, if $n=5$, $i=4$ and
$j=6$, we are projecting onto $\Span(e_4, e_5, e_1)$.) Note that,
when $j \geq i+n-1$, we project onto all of $[n]$. If $V$ is the row
span of a $k \times n$ matrix $M$, then $r_{ij}(V)$ is the rank of
the submatrix of $M$ consisting of columns $i$, $i+1$, \dots, $j$.

\subsection{Positroid varieties and open positroid varieties}
\begin{lem}\label{L:rankcyclic}
Let $V \in \Gr(k,n)$. Then $r_{\bullet \bullet}(V)$ is a cyclic
rank matrix of type $(k,n)$.
\end{lem}
\begin{proof}
Conditions (C1'), (C2'), and (C5) are clear from the definitions.
Let $M$ be a $k \times n$ matrix whose row span is $V$; let $M_i$ be the $i^{\textrm{th}}$ column of $M$.
Condition (C3) says that adding a column to a matrix either preserves the rank of that matrix or increases it by one.
The hypotheses of condition (C4) state that $M_{i}$ and $M_{j}$ are in the span of $M_{i+1}$, $M_{i+2}$, \dots, $M_{j-1}$; the conclusion is that
$\dim \Span(M_i, M_{i+1}, \dots,M_{j-1}, M_j) = \dim \Span(M_{i+1}, \dots, M_{j-1})$.
\end{proof}

For any cyclic rank matrix $r$, let $\Pio_r$ be the subset of $\Gr(k,n)$ consisting of those $k$-planes $V$ with cyclic rank matrix $r$.
We may also write $\Pio_f$, $\Pio_{\J}$ or $\Pio_u^w$ where $f$ is the bounded affine permutation, $\J$ the juggling pattern or $\langle u, w \rangle$ the equivalence class of $k$-Bruhat interval corresponding to $r$.

The next result follows directly from the definitions.
Recall that $\chi = [23\cdots (n-1)n1] \in S_n$ denotes the long cycle.

\begin{lem} \label{lem:OpenCyclic}
For any $\J = (J_1, J_2, \ldots, J_n) \in \Jugg(k,n)$, we have
$$
\Pio_{\J} = \Xo_{J_1} \cap \chi(\Xo_{J_2}) \cap \cdots \cap
\chi^{n-1}(\Xo_{J_n}).
$$
\end{lem}

By Lemma \ref{L:rankcyclic} and our combinatorial  bijections, we
have
$$
\Gr(k,n) = \coprod_{\J \in \Jugg(k,n)} \Pio_{\J}.
$$
We call the sets $\Pio_{\J}$ \defn{open positroid varieties}.
Postnikov \cite{Pos} showed that $\Pio_{\J}$ (and even $(\Pio_{\J})_{\geq 0}$)
is non-empty if $\J \in 
\Jugg(k,n)$ (this statement also follows from Proposition
\ref{P:triplecyclicrank} below).  We define the
\defn{positroid varieties} $\Pi_{\J}$ to be the closures $\Pi_{\J}
:= \overline{\Pio_{\J}}$.

\subsection{From $\Q(k,n)$ to cyclic rank matrices}

We now describe a stratification of the Grassmannian due to
Lusztig~\cite{Lus}, and further studied by Rietsch~\cite{Rie}.
(This stratification was also independently discovered by Brown, Goodearl
and Yakimov~\cite{BGY, GY}, motivated by ideas from Poisson geometry;
we will not discuss the Poisson perspective further in this paper.)
Lusztig and Rietsch's work applies to any partial flag variety, and we
specialize their results to the Grassmannian.

The main result of this section is the following:

\begin{prop}\label{P:triplecyclicrank}
  Let $u \leq_k w$, and
  $f_{u,w}$ be the corresponding affine
  permutation from \S \ref{S:pairstobounded}.  Recall that
  $\Xo_u^w$ denotes the open Richardson variety in $\Fl(n)$ and $\pi$ the map
  $\Fl(n) \to \Gr(k,n)$.  Then   $\Pio_f = \pi(\Xo_u^w)$.
\end{prop}


\begin{remark}
If $u \leq w$, but  $u \not \leq_k w$, then $\pi(\Xo_u^w)$ may not
be of the form $\Pio_f$.  See \cite[Remark 3.5]{KLS}.
\end{remark}
%



The projection $\pi(\Xo_{u}^{w})$ depends only on the equivalence class of $[u,w]_k$ in $\Q(k,n)$.

\begin{prop}[{\cite[Lemma 3.1]{KLS}}] \label{P:EquivalenceClass}
Suppose $[u,w]_k \sim [u',w']_k$ in $\Q(k,n)$. Then
$\pi(\Xo_{u}^{w}) = \pi(\Xo_{u'}^{w'})$.
\end{prop}

We now introduce a piece of notation which will be crucial in the proof of Proposition~\ref{P:triplecyclicrank}, but will then never appear again.
Let $V \in \Gr(k,n)$.  Given a flag $F_\bullet$ in $\C^n$, we obtain
another flag $F_\bullet(V)$ containing $V$ as the $k$th subspace,
as follows.  Take the sequence $F_0 \cap V, F_1 \cap V, \ldots, F_n
\cap V$ and remove repetitions to obtain a partial flag $F_\bullet \cap V$
inside $\C^n$, with dimensions $1$, $2$, \dots $k$.  Next take the sequence $V+F_0$, $V+F_1$,
$V+F_2$, \dots, $V+F_n$ and remove repetitions to obtain a partial
flag $F_\bullet + V$ inside $\C^n$ of dimensions $k$, $k+1$, \dots, $n$.
Concatenating $F_\bullet \cap V$ and $F_\bullet + V$ gives a flag $F_\bullet(V)$ in $\C^n$.  The
flag $F_\bullet(V)$ is the ``closest'' flag to $F_\bullet$ which
contains $V$ as the $k$th subspace.
This notion of ``closest flag'' is related to the notion of ``closest Borel subgroup'' in~\cite[\S 5]{Rie}, and many of our arguments are patterned on arguments of~\cite{Rie}.

%

\begin{lemma} \label{lem:GrassSchubert}
Let $w$ be a Grassmannian permutation.
Let $F_{\bullet}$ be a complete flag and let $V = F_k$.
Then $F_{\bullet} \in \Xo^{w}$ if and only if
\begin{enumerate}
 \item $V \in \Xo^{\sigma(w)}$
and
\item  $F_{\bullet} = E_{\bullet}(V)$.
\end{enumerate}
\end{lemma}

\begin{proof}
The flag $F_{\bullet}$ is in $\Xo^{w}$ if and only if, for every $i$ and $j$,
\[
\dim(F_i \cap E_j) = \# \left( w([i]) \cap [j] \right) \mbox{ or, equivalently, } \dim(F_i + E_j) = (i+j) - \# \left( w([i]) \cap [j] \right). \label{InSchubert}
\]
When $i=k$, the equation above 
is precisely the condition that $V \in \Xo^{\sigma(w)}$.
Therefore, when proving either direction of the equivalence, we may assume that $V \in \Xo^{\sigma(w)}$.

Since $V \in \Xo^{\sigma(w)}$,
$$E_{\bullet}(V) = \left( E_{w(1)} \cap V, E_{w(2)} \cap V, \ldots, E_{w(k)} \cap V, E_{w(k+1)}+V, E_{w(k+2)}+V, \ldots, E_{w(n)}+V \right).$$

Let $i \leq k$. If $F_{\bullet} \in  \Xo^{w}$, then $\dim (F_i \cap E_{w(i)}) = \# \left( w([i]) \cap [w(i)] \right) = i$, where we have used that $w$ is Grassmannian.
However, $F_i \cap E_{w(i)} \subseteq V \cap E_{w(i)}$, which also has dimension $i$ because $V \in \Xo^{\sigma(w)}$.
So $F_{\bullet} \in  \Xo^{w}$ implies that $F_i = V \cap E_{w(i)}$.
Similarly, for $i > k$, the equation $\dim (F_i + E_{w(i)}) = i+w(i) - \# \left( w([i]) \cap [w(i)] \right)$
implies that $F_i = E_{w(i)}+V$.
So, if $F_{\bullet} \in \Xo^{x}$ then $F_{\bullet} = E_{\bullet}(V)$.
The argument is easily reversed.
\end{proof}

%
%
%

\begin{lemma} \label{lem:RietschMainPoint}
Let $w$ be a Grassmannian permutation, with $u \leq w$ and let $V \in \Gr(k,n)$.
Then $V \in \pi(\Xo_{u}^{w})$ if and only if
\begin{equation}
E_{\bullet}(V)_i= V \cap E_{w(i)} \mbox{ for $1 \leq i \leq k$} \label{E:RMP1}
\end{equation}
\begin{equation}
E_{\bullet}(V)_i =V + E_{w(i)} \mbox{ for $k < i \leq n$} \label{E:RMP2}
\end{equation}
and
\begin{equation}
E_{\bullet}(V) \in \Xo_{u}. \label{E:RMP3}
\end{equation}

\end{lemma}

\begin{proof}
By definition, $V \in \pi(\Xo_{u}^{w})$ if and only if there is a flag $F_{\bullet}$ with $V= F_k$ and $F_{\bullet} \in \Xo_{u} \cap \Xo^{w}$.
By Lemma~\ref{lem:GrassSchubert}, this flag $F_{\bullet}$, should it exist, must be $E_{\bullet}(V)$.
By Lemma~\ref{lem:GrassSchubert}, $E_{\bullet(V)}$ lies in $\Xo^{w}$ if and only if $V$ lies in $\Xo^{\sigma(w)}$.
So $E_{\bullet(V)} \in  \Xo_{u} \cap \Xo^{w}$ if and only if $V \in \Xo^{\sigma(w)}$ and $E_{\bullet}(V) \in \Xo_{u}$.
Now, conditions~(\ref{E:RMP1}) and~(\ref{E:RMP2}) determine the dimension of $V \cap E_{w(i)}$ for all $i$.
They are precisely the condition on $\dim \left( V \cap w(j) \right)$ occuring in the definition of $\Xo^{\sigma(w)}$.
So conditions~(\ref{E:RMP1}), (\ref{E:RMP2}) and~(\ref{E:RMP3}) are equivalent to the condition that $V \in \Xo^{\sigma(w)}$ and $E_{\bullet}(V) \in \Xo_{u}$.
\end{proof}

\begin{proof}[Proof of Proposition~\ref{P:triplecyclicrank}]
First, note that by Proposition~\ref{P:EquivalenceClass}, and the
observation that $f_{u,w}$ only depends on the equivalence class
$\langle u,w \rangle$ in $\Q(k,n)$, we may replace $(u,w)$ by any
equivalent pair in $\Q(k,n)$.  We may thus assume that $w$ is
Grassmannian (Proposition~\ref{prop:GrassRep}). By Lemma~\ref{lem:RietschMainPoint}, $V \in
\pi(\Xo_{u}^{w})$ if and only if conditions~(\ref{E:RMP1}),  (\ref{E:RMP2}) and~(\ref{E:RMP3}) hold.

Suppose that $V \in \pi(\Xo_{u}^{w})$.
Let $r=r_{\bullet \bullet}(V)$.
Let $a \in \Z$ and let $b = f_{u,w}(a)$; without loss of generality we may assume that $1 \leq a \leq n$.
Set $i = w^{-1}(a)$.
We now check that $(a,b)$ is a special entry of $r$.

\textbf{Case 1:} $i \not \in [k]$.
In this case $f_{u,w}(a) = u t_{\omega_k} w^{-1}(a) \in [n]$.
Since $f_{u,w} \in \Bound(k,n)$, we deduce that $a \leq b \leq n$.
(Occasionally, our notation will implicitly assume $a<b$, we leave it to the reader to check the boundary case.)
By conditions~(\ref{E:RMP2}) and~(\ref{E:RMP3}),
$$\dim \Project_{[b]} \left( V + E_a \right) =\# \left( u([i]) \cap [b] \right).$$
We can rewrite this as
$$\dim \left( V + E_a + w_0 E_{n-b} \right) = (n-b) +\# \left( u([i]) \cap [b] \right)$$
or, again,
$$\dim \Project_{[a+1,b]} (V) = \# \left( u([i]) \cap [b] \right)-a. $$
(We have used $a \leq b \leq n$ to make sure that $\dim  E_a + w_0 E_{n-b} = n-b+a$.)
In conclusion,
$$r_{(a+1)b} = \# \left( u([i]) \cap [b] \right) -a.$$
A similar computation gives us
$$r_{(a+1)(b-1)} = \# \left( u([i]) \cap [b-1] \right) -a.$$

We now wish to compute $r_{ab}$ and $r_{a(b-1)}$.
This time, we have $V+E_{a-1} = E_{\bullet}(V)_{i-1}$.
So we deduce from condition~(\ref{E:RMP2}) that
$$r_{ab} = \# \left( u([i-1]) \cap [b] \right) - (a-1)$$
and
$$r_{a(b-1)} = \# \left( u([i-1]) \cap [b-1] \right) - (a-1).$$

Now, $u(i) = u(w^{-1}(a))=b$.
So, $r_{(a+1)b} - r_{(a+1)(b-1)}=1$ and, since $b \not \in u([i-1])$, we also have $r_{ab} - r_{a(b-1)}=0$.
So $(a,b)$ is special as claimed.

\textbf{Case 2:} $i \in [k]$. 
In this case, $b=u(i)+n$ and $n+1 \leq b \leq a+n$.
We mimic the previous argument, using $V \cap E_a \cap w_0 E_{2n-b}$ in place of $V+E_a+w_0 E_{n-b}$, the conclusion again is that $(a,b)$ is a special entry of $r$.

We have now checked, in both cases, that $(a,b)$ is a special entry of $r$.
Therefore, the affine permutation $g$ associated to $r$ has $g(a)=b$.
Since $f_{u,w}(a)=b$, we have checked that $f_{u,w}=g$.
We have thus shown that, if $V \in \pi(\Xo_{u}^{w})$, then $V \in \Pio_{f_{u,w}}$.

We now must prove the converse. Let $r_{\bullet \bullet}(V) =
r(f_{u,w})$. Let $(u', w')$ be such that $E_{\bullet}(V) \in
\Xo_{u'}^{w'}$, so we know that $V \in \pi(\Xo_{u'}^{w'})$. By
Lemma~\ref{lem:GrassSchubert}, $w'$ is Grassmannian. So $r_{\bullet
\bullet}(V) = r(f_{u',w'})$ and $f_{u',w'}=f_{u,w}$. However, by
Proposition~\ref{P:pairstobounded}, this shows that $[u,w]_k$ and
$[u',w']_k$ represent the same element of $\Q(k,n)$. Since $w$ and
$w'$ are both Grassmannian, this means that $u = u'$ and $w = w'$,
and $V \in \pi(\Xo_{u}^{w})$ as desired.
\end{proof}

\subsection{Positroid varieties are projected Richardson varieties}
\label{ssec:projectedRichardsons}

Lusztig~\cite{Lus} exhibited a stratification $\coprod P_{(u,v,w)}$
of $\Gr(k,n)$ indexed by triples $(u,v,w) \in S_{n,k}^{\max} \times (S_k \times S_{n-k})
\times \GS$ satisfying $u \leq wv$, and showed that his strata satisfy
$$P_{(u,v,w)} = \pi(\Xo^{wv}_{u}) = \pi(\Xo^{w}_{uv^{-1}}).$$
Furthermore, the projection $\pi: \Fl(n) \to \Gr(k,n)$ restricts to
an isomorphism on $\Xo^{wv}_{u}$.  Using the bijection between the
triples $(u,v,w)$ and $\Q(k,n)$ (see \S \ref{sec:kBruhat}), it
thus follows from Proposition \ref{P:triplecyclicrank} that

\begin{thm}\label{thm:projectedRichardsons}
  The stratification of $\Gr(k,n)$ by open positroid
  varieties is identical to Lusztig's stratification.
  If $f = f_{u,w}$ corresponds to $\langle u, w \rangle$ under the
  bijection $\Q(k,n) \to \Bound(k,n)$ of \S \ref{S:pairstobounded},
  then $\pi(\Xo_u^w) = \Pio_f$. The varieties $\Pi_f$ and $\Pio_f$ are
  irreducible of codimension~$\ell(f)$, and $\Pio_f$ is smooth.
  For any Richardson variety $X_u^w$, whether or not $u \leq_k w$, the projection $\pi(X_u^w)$ is a closed positroid variety.
\end{thm}

\begin{proof}
  Open Richardson varieties in the flag manifold are smooth and irreducible
  (by Kleiman transversality).
Lusztig's strata are, by definition, the projected open Richardson varieties, which we have just showed are the same as the open positroid varieties. Lusztig shows that $\pi$ restricted to $\Xo_u^w$ is an isomorphism on its image, so $\dim \Pi_f=\dim X_u^w = \ell(w) - \ell(u)$ and $\Pi_f$ is irreducible. By Theorem~\ref{T:pairsboundedposet}, $\ell(w)-\ell(u) = k(n-k) -\ell(f)$, so $\Pio_f$ has codimension $\ell(f)$, as does its closure $\Pi_f$. 

See \cite[Proposition 3.3]{KLS} for the fact that the projection of any Richardson $X_{u}^w$ is equal to the projection of some $X_{u'}^{w'}$ with $u'$ anti-Grassmannian.
\end{proof}
For $u \leq_k w$, we shall call $X^w_u$ a \defn{Richardson model}
for $\Pi_{f_{u,w}}$.  We refer the reader to \cite{KLS} for a discussion of projections of closed Richardson varieties.  In particular, for any $u\leq w$ (not necessarily a $k$-Bruhat relation) there exists a bounded affine permutation $f$ such that $\pi(X_u^w) = \Pi_f$.



Postnikov \cite{Pos} parametrized the ``totally nonnegative part''
of any open positroid variety, showing that it is homeomorphic to an
open ball. Before one knows that positroid varieties are actually
irreducible, one can use this parametrizability to show that only one
component intersects the totally nonnegative part of the
Grassmannian. A priori there might be other components, so it is
nice to know that in fact there are not.


We now describe the containments between positroid varieties:

\begin{thm}\label{thm:intersectSchubs}
  Open positroid varieties form a stratification of the Grassmannian.
  Thus for $f \in \Bound(k,n)$ we have
  $$
  \Pi_f = \coprod_{f' \geq f} \Pio_{f'} = X_{J_1} \cap
  \chi(X_{J_2}) \cap \cdots \cap \chi^{n-1}(X_{J_n}).
  $$
  where $(J_1,J_2,\ldots,J_n) \in \Jugg(k,n)$ corresponds to $f$.
\end{thm}

\begin{proof}
Rietsch \cite{Rie} described the closure relations of Lusztig's
stratification of partial flag varieties; see also \cite{BGY}. The first equality is Rietsch's result, translated from the language of $\Q(k,n)$ to $\Bound(k,n)$.

We know that $X_J = \coprod_{I \geq J} \Xo_I$. Using this to expand the
intersection $X_{J_1} \cap \chi(X_{J_2}) \cap \cdots \cap \chi^{n-1}(X_{J_n})$
and applying Lemma~\ref{lem:OpenCyclic} gives the second equality.
\end{proof}

We note that Postnikov \cite{Pos} also described the same closure
relations for the totally nonnegative Grassmannian, using Grassmann
necklaces and decorated permutations.

For a matroid $\M$ let
$$
\GGMS(\M) = \{V \in \Gr(k,n) \mid \Delta_I(V) \neq 0
\Longleftrightarrow I \in M \}
$$
denote the GGMS stratum of the Grassmannian \cite{GGMS}.  Here for $I
\in \binom{[n]}{k}$, $\Delta_I$ denotes the Pl\"{u}cker coordinate
labeled by the columns in $I$.  Recall that in Section
\ref{S:positroids}, we have defined the positroid envelope of a
matroid.  It is easy to see that
$$
\Pio_f = \coprod_{\M :\ \J(\M) = \J(f)} \GGMS(\M).
$$

\begin{prop}\label{P:GGMSdense}
Let $\M$ be a positroid.  Then $\GGMS(\M)$ is dense in
$\Pio_{\J(\M)}$.
\end{prop}
\begin{proof}
Suppose $f \in \Bound(k,n)$ is such that $\J(f) = \J(\M)$. Postnikov
\cite{Pos} showed that the totally nonnegative part $\GGMS(\M)_{
\geq 0}$ of $\GGMS(\M)$ is a real cell of dimension $k(n-k) -
\ell(f)$. Thus $\GGMS(\M)$ has at least dimension $k(n-k) -
\ell(f)$. By Theorem~\ref{thm:projectedRichardsons} $\Pio_{f}$ is
irreducible with the same dimension.  It follows that $\GGMS(\M)$ is
dense in $\Pio_{\J(\M)}$.
\end{proof}

\begin{cor}\label{cor:pluckerdefined}
Let $\M$ be a positroid.  Then as sets,
$$
\Pi_{\J(\M)} = \overline{\GGMS(\M)} = \{V \in \Gr(k,n) \mid I \notin
\M \Rightarrow \Delta_I(V) = 0 \}.
$$
\end{cor}

\begin{proof}
  The first equality follows from Proposition \ref{P:GGMSdense}. The
  second follows from Theorem~\ref{thm:intersectSchubs} and the
  description of Schubert varieties by vanishing of Pl\"{u}cker coordinates:

\hfill
$
X_J = \{V \in \Gr(k,n) \mid I < \sigma(J) \Rightarrow
\Delta_I(V) = 0 \}.
$\hfill
\end{proof}

Consider the set on the right hand side of the displayed equation in Corollary~\ref{cor:pluckerdefined}. Lauren Williams conjectured that this set was irreducible; this now follows from Corollary~\ref{cor:pluckerdefined} and Theorem~\ref{thm:projectedRichardsons}.

\subsection{Geometric properties of positroid varieties}
The following results follow from Theorem \ref{thm:projectedRichardsons} and the geometric results of \cite{KLS}.

\begin{thm}\label{thm:geometry}
Positroid varieties are normal, Cohen-Macaulay, and have rational singularities.
\end{thm}

\begin{thm}\label{thm:Frobenius}
  There is a Frobenius splitting on the Grassmannian that compatibly
  splits all the positroid varieties therein.  Furthermore, the set of positroid varieties is exactly the set of compatibly split subvarieties of the Grassmannian.
\end{thm}

\begin{theorem} \label{T:LinearGeneration}
  Let $\M$ be a positroid. Then the ideal defining the variety $\Pi_{\J(\M)}$ inside $\Gr(k,n)$
  is generated by the Pl\"ucker coordinates $\{\Delta_I : I \notin \M\}$.
\end{theorem}

\begin{proof}
  By Theorem \ref{thm:intersectSchubs}, $\Pi_{\J(\M)}$ is the
  set-theoretic intersection of some permuted Schubert varieties.
  By the Frobenius splitting results of \cite{KLS}, it is
  also the scheme-theoretic intersection.

  Hodge proved that Schubert varieties (and hence permuted   Schubert varieties) are defined by the vanishing of Pl\"ucker
  coordinates. (See e.g. \cite{Ramanathan}, where a great
  generalization of this is proven using Frobenius splitting.)
  The intersection of a family of them is defined by the   vanishing of all their individual coordinates.

  As explained in \cite[Proposition 3.4]{FZrec},
  it is easy to determine which Pl\"ucker coordinates vanish on
  a $T$-invariant subscheme $X$ of the Grassmannian;
  they correspond to the fixed points not lying in $X$.
\end{proof}

\begin{cor} \label{cor:QuadGeneration}
 Let $M$ be a positroid. Embed $\Gr(k,n)$ into $\PP^{\binom{n}{k}-1}$ by the Pl\"ucker embedding.
Then the ideal of $\Pi_{\J(M)}$ in $\PP^{\binom{n}{k}-1}$ is
generated in degrees $1$ and $2$.
\end{cor}

\begin{proof}
By Theorem~\ref{T:LinearGeneration}, the ideal of $\Pi_{\J(M)}$ is
the sum of a linearly generated ideal and the ideal of $\Gr(k,n)$.
It is classical that the ideal of $\Gr(k,n)$ is generated
in degree $2$.
\end{proof}

\begin{remark}\label{rem:notPluckerDefined}
  For a subvariety $X\subseteq G/P \subseteq \PP V$ of a general flag
  manifold embedded in the projectivization $\PP V$ of an irreducible
  representation, one can ask whether $X$ is defined as a set by the
  vanishing of extremal weight vectors in $V$.  This is easy to show
  for Schubert varieties (see \cite{FZrec}) and more generally for
  Richardson varieties.

Since the above collorary proves this property for positroid varieties,
and \cite{FZrec} prove it for Richardson varieties in $G/B$,
one might conjecture that it would be true for
projected Richardson varieties in other $G/P$s.
This is {\em not} the case:
consider the Richardson variety $X_{1324}^{4231}$ projecting
to a divisor in the partial flag manifold $\{(V_1 \subset V_3 \subset \complexes^4)\}$.
One can check that the image contains every $T$-fixed point,
so no extremal weight vector vanishes on it.

For any irreducible $T$-invariant subvariety $X \subseteq G/P$,
the set of $T$-fixed points $X^T \subseteq (G/P)^T \iso W/W_P$
forms a {\em Coxeter matroid} \cite{CoxeterMatroids},
and $X$ is contained in the set where the extremal weight vectors
corresponding to the complement of $X^T$ vanish. If the containment is
proper, as in the above example, one may take this as evidence that the
Coxeter matroid is not a good description of $X$.
We saw a different knock against matroids in Remark \ref{rem:SashaParam}.
\end{remark}

\section{Examples of positroid varieties}
\label{sec:examples}

In this section, we will see that a number of classical objects
studied in algebraic geometry are positroid varieties, or closely
related to positroid varieties.

First, for any $I \in \binom{[n]}{k}$, the Schubert variety
$X_{I}$ in the Grassmannian is the positroid variety associated
to the positroid $\{ J : J \geq I \}$.  Similarly, the cyclically permuted Schubert varieties $\chi^i \cdot X_{I}$ are also positroid varieties.
Similarly, the Richardson varieties $X_{I}^{K}$ are positroid varieties, corresponding to the positroid
$\{ J : I \leq J \leq K \}$.

\newcommand{\GL}{\mathrm{GL}}
\newcommand{\Mat}{\mathrm{Mat}}
\newcommand{\Id}{\mathrm{Id}}

Another collection of objects, closely related to Schubert
varieties, are the {\em graph Schubert varieties}.  Let $X_w$ be a
Schubert variety in $\Fl(n)$.  Considering $\Fl(n)$ as $B_{-}
\backslash \GL_n$ (where $B_{-}$ is the group of invertible lower
triangular matrices), let $X'_w$ be the preimage of $X_w$ in
$\GL_n$. The \defn{matrix Schubert variety} $MX_w$, introduced in
\cite{Fulton}, is the closure of $X'_w$ in $\Mat_{n \times n}$.
$MX_w$ is cut out of $\Mat_{n \times n}$ by imposing certain rank
conditions on the top-left justified submatrices (as was explained
in \S \ref{ssec:SchubertRichardsonInFlag}). Embed $\Mat_{n
\times n}$ into $\Gr(n,2n)$ by the map $\Gamma$ which sends a matrix
$M$ to the graph of the linear map $\vec v \mapsto M\vec v$; its
image is the \defn{big cell} $\{ \Delta_{[n]} \neq 0 \}$. In
coordinates, $\Gamma(M)$ is the row span of the $n \times 2n$ matrix
$\left[ \begin{smallmatrix} \Id & M \end{smallmatrix} \right]$. We
will abuse notation by also calling this matrix $\Gamma(M)$. We
introduce here the \defn{graph Schubert variety}, $GX_w$, as the
closure of $\Gamma(MX_w)$ in $\Gr(n,2n)$.  Graph Schubert varieties
will be studied further in a separate paper by the first author,
\cite{Knutson}.

Let us write $M_{[1,i],[1,j]}$ for the top-left $i \times j$ submatrix
of $M$.  Then the rank of $M_{[1,i], [1,j]}$ is $n-i$ less than the
rank of the submatrix of $\Gamma(M)$ using rows $\{i+1, i+2, \ldots,
n, n+1, \ldots, n+j \}$.  So every point of $GX_w$ obeys certain rank
bounds on the submatrices of these types.  These rank bounds are
precisely the rank bounds imposed by $r(f)$, where $f$ is the affine
permutation $f(i)=w(i)+n$ for $1 \leq i \leq n$, $f(i)=i+n$ for
$n+1 \leq i \leq 2n$.  So $GX_w$ is contained in $\Pi_f$, with
equality on the open set $\Gamma(\Mat_{n \times n})$.  But $\Pi_f$ and $GX_w$
are both irreducible, so this shows that $GX_w=\Pi_f$.
In \S \ref{sec:cohomology}, we will see that cohomology classes of general
positroid varieties will correspond to affine Stanley symmetric
functions; under this correspondence, graph Schubert varieties give
the classical Stanley symmetric functions.

The example of graph Schubert varieties can be further generalized
\cite[\S 0.7]{BGY}. Let $u$ and $v$ be two elements of $S_n$
and consider the affine permutation $f(i)=u(i)+n$ for $1 \leq i \leq
n$, $f(i)=v^{-1}(i-n)+2n$ for $n+1 \leq i \leq 2n$.  (So our
previous example was when $v$ is the identity.)  Let us look at
$\Pi_f \cap \Gamma(\Mat_{n \times n})$. This time, we impose
conditions both on the ranks of the upper left submatrices and the
lower right submatrices.  In fact, $\Pio_f$ lies entirely within
$\Gamma(\GL_n)$ and is $\Gamma \left( B_{-} u B_{+} \cap B_{+} v
B_{-} \right)$.  This is essentially Fomin and Zelevinsky's \cite{FZdouble} \defn{double Bruhat cell}.
Precisely,  the double Bruhat cell $\GL_n^{u,v}$ is $B_{+} u B_{+} \cap
B_{-} v B_{-}$. So the positroid variety $\Pi_{f}$ is the closure in
$\Gr(n,2n)$ of $\Gamma(w_0 \GL_n^{w_0 u, w_0 v})$.

Finally, we describe a connection of positroid varieties to quantum
cohomology, which we discuss further in \S \ref{sec:QH}. For
$C$ any algebraic curve in $\Gr(k,n)$, one defines the degree of $C$
to be its degree as a curve embedded in $\PP^{\binom{n}{k}-1}$ by
the Pl\"ucker embedding; this can also be described as 
$\int_{\Gr(k,n)} [C] \cdot [X_{\Box}]$ where 
$X_{\Box}$ is the Schubert divisor.  Let $I$,
$J$ and $K$ be three elements of $\binom{n}{k}$ and $d$ a
nonnegative integer, $d \leq k$, such that $\codim X_I + \codim X_J
+ \codim X_K = k(n-k) + d n$.

Intuitively, the (genus
zero) quantum product $\langle X_{I} X_{J} X_{K} \rangle_d$
is the number of curves in $\Gr(k,n)$, of genus zero and degree $d$,
which meet $X_{I}(F_{\bullet})$, $X_{J}(G_{\bullet})$ and
$X_{K}(H_{\bullet})$ for a generic choice of flags $F_{\bullet}$, $G_{\bullet}$ and
$H_{\bullet}$.  This is made precise via the construction of spaces of stable
maps, see \cite{FP}.

Define $E(I,J,d)$ to be the space of degree $d$ stable maps of a genus
zero curve with three marked points to $\Gr(k,n)$, such that the first
marked point lands in $X^I$ and the second marked point lands in
$X_J$. Let $S(I,J,d)$ be the subset of $\Gr(k,n)$ swept out by the
third marked point. It is intuitively plausible that 
$\langle X_{I} X_{J} X_{K} \rangle_d$ is $\int_{\Gr(k,n)} [S(I,J,d)]\cdot [X_K]$
and we will show that, under certain hypotheses, this holds.  We will
show that (under the same hypotheses) $S(I,J,d)$ is a positroid variety.

%

%

%



%

%

\section{The cohomology class of a positroid variety}
\label{sec:cohomology}

Let $H^*(\Gr(k,n)), H^*_T(\Gr(k,n))$
denote the ordinary and equivariant 
(with respect to the natural action of $T = (\C^*)^n$) cohomologies of the
Grassmannian, with integer coefficients.
If $X \subset \Gr(k,n)$ is a $T$-invariant subvariety of the
Grassmannian, we let $[X]_0 \in H^*(\Gr(k,n))$ denote its ordinary
cohomology class, and $[X] \in H^*_T(\Gr(k,n))$ denote its
equivariant cohomology class. We also write $[X]|_p$ for the
restriction of $[X]$ to a $T$-fixed point $p$.  We index the fixed
points of $\Gr(k,n)$ by $\binom{[n]}{k}$. We use similar notation
for the flag manifold $\Fl(n)$, whose fixed points are
indexed by $S_n$. 
Recall that $\pi : \Fl(n) \to \Gr(k,n)$ denotes the ($T$-equivariant)
projection.

In \cite{Lam1}, a symmetric function $\tF_f \in \Sym$ is introduced
for each affine permutation $f$.  Let $\psi: \Sym \to H^*(\Gr(k,n))$
denote the natural quotient map.  In this section, we show

\begin{thm}\label{thm:affineStanley}
    Let $f \in \Bound(k,n)$.  Then $\psi(\tF_f) = [\Pi_f]_0 \in
    H^*(\Gr(k,n))$.
\end{thm}

\subsection{Monk's rule for positroid varieties}
The equivariant cohomology ring $H^*_T(\Gr(k,n))$ is a module over
$H^*_T(\pt) = \Z[y_1,y_2,\ldots,y_n]$.  The ring $H^*_T(\Gr(k,n))$
is graded with the real codimension, so that $\deg(y_i) = 2$ and
$\deg([X]_T) = 2\ \mathrm{codim}(X)$ for an irreducible
$T$-equivariant subvariety $X \subset \Gr(k,n)$.

Let $X_\square \in
H^*_T(\Gr(k,n))$ denote the class of the Schubert divisor.  Note
that $\pi^*(X_\square) \in H^*_T(\Fl(n))$ is the class $[X_{s_k}]$ of the
$k$th Schubert divisor.
We recall the equivariant Monk's formula (see for example \cite{KK}):
\begin{equation}\label{E:monk}
[X_{s_k}].[X_w]= ([X_{s_k}]|_w).[X_w] + \sum_{w \lessdot_k v} [X_v].
\end{equation}

\begin{prop}\label{P:monk}
Let $\Pi_f$ be a positroid variety with Richardson model
$X_u^{w}$.  Then
\begin{equation}\label{E:posmonk}
X_\square\cdot[\Pi_f] = (X_\square|_{\sigma(u)})\cdot[\Pi_f] + \sum_{u \lessdot_k u'
\leq_k w} [\Pi_{f_{u',w}}].
\end{equation}
\end{prop}

Here $\sigma$ is the map $\sigma_k : S_n \to \binom{[n]}{k}$.

\begin{proof}
Let $X_u^{w}$ be a Richardson model for $\Pi_f$.  Then, using the
projection formula and \eqref{E:monk}, we have in $H^*_T(\Gr(k,n))$,
\begin{align*}
      X_\square \cdot[\Pi_f] &= \pi_*( \pi^*(X_\square)\cdot[X^{w}]\cdot[X_u] ) \\
    &= ([X_{s_k}]|_u)\cdot[\Pi_f] + \pi_*( \sum_{u \lessdot_k u'} [X_{u'}]\cdot[X^{w}]) \\
    &= (X_\square|_{\sigma(u)})\cdot[\Pi_f] + \sum_{u \lessdot_k u'} \pi_*(    [X_{u'}^{w}]).
   \end{align*}
But
$$\pi_*( [X_{u'}^{w}]) = \begin{cases} [\Pi_{f_{u',w}}] & \mbox{if $u' \leq_k w$,} \\
0 & \mbox{otherwise.} \end{cases}
$$
\end{proof}

\begin{cor}\label{C:hypsection}
  Let $\Pi_f$ be a positroid variety with Richardson model
  $X_u^{w}$, and let $X_\square \subseteq \Gr(k,n)$ denote the Schubert divisor.
  Then as a scheme,
  $$ (u\cdot X_\square) \cap \Pi_f
  = \bigcup_{u \lessdot_k u' \leq_k w} \Pi_{f_{u',w}}. $$
\end{cor}

\begin{proof}
  The containment $\supseteq$ follows from Theorem \ref{T:LinearGeneration}.
  The above Proposition tells us that the two sides have the
  same cohomology class, hence any difference in scheme structure must occur in
  lower dimension; this says that $(u\cdot X_\square) \cap \Pi_f$
  is generically reduced (and has no other top-dimensional components).
  But since $\Pi_f$ is irreducible and normal
  (Theorem \ref{thm:projectedRichardsons} and Theorem \ref{thm:geometry}),
  a generically reduced hyperplane section of it must be equidimensional
  and reduced.
\end{proof}


\begin{lem}\label{L:recursive}
The collection of positroid classes $[\Pi_f] \in H^*_T(\Gr(k,n))$
are completely determined by:
\begin{enumerate}
\item
$[\Pi_f]$ is homogeneous with degree $\deg([\Pi_f]) = 2\ell(f)$,
\item
Proposition \ref{P:monk}, and
\item the positroid point classes
$\left\{ [\Pi_{t_{w.\omega_k}}] = [\sigma(w)] \mid w \in \GS \right\}$.
\end{enumerate}
\end{lem}
\begin{proof}
Let $f \in \Bound(k,n)$.  We may assume by induction that the
classes $[\Pi_{f'}]$ for $\ell(f')> \ell(f)$ have all been
determined.  The case $\ell(f) = k(n-k)$ is covered by assumption
(3), so we assume $\ell(f) < k(n-k)$.  Using Proposition
\ref{P:monk}, we may write
$$
(X_\square-X_\square|_{\sigma(v)}).[\Pi_f] = \sum_{v \lessdot_k v' \leq_k w}
[\Pi_{f_{(v',w)}}].
$$
Now, the class $X_\square-X_\square|_{\sigma(v)}$ does not vanish when restricted to
any fixed point $J \neq \pi(v)$ (see \cite{KT}), so the above
equation determines $[\Pi_f]|_J$ for every $J \neq \pi(v)$. Thus
if $a$ and $b$ are two classes in $H^*_T(\Gr(k,n))$ satisfying
\eqref{E:posmonk}, then $a - b$ must be supported on
$\pi(v)$.  This means that $a - b$ is a multiple of the
point class $[\pi(v)]$.  But $\deg([\pi(v)]) = 2k(n-k)$ and
$\deg(a) = \deg(b) = \ell(f) < 2k(n-k)$ so
$a = b$.  Thus $[\Pi_f]$ is determined by the three
assumptions.
\end{proof}

\subsection{Chevalley formula for the affine flag variety}
Let
$\AFl(n)$ denote the affine flag variety of $GL(n,\C)$. We let
$\{\xi^f \in H^*_T(\AFl(n))\mid f \in \tS_n \}$ denote the
equivariant Schubert classes, as defined by Kostant and Kumar in
\cite{KK}.

Now suppose that $f \in \tS_n$.  We say that $f$ is \defn{affine
Grassmannian} if $f(1) < f(2) < \cdots < f(n)$.  For any $f \in
\tS_n$, we write $f^0 \in \tS_n$ for the affine permutation given by
$f^0 = [\cdots g(1)g(2) \cdots g(n) \cdots]$ where $g(1), g(2),
\cdots, g(n)$ is the increasing rearrangement of $f(1), f(2),
\cdots, f(n)$. Then $f^0$ is affine Grassmannian.  Suppose that $f
\lessdot g$ and $f^0 \neq g^0$. Then we say that $g$
\defn{0-covers} $f$ and write $f \lessdot_0 g$.
These affine analogues of $k$-covers were studied in \cite{LLMS}.

For a transposition $(ab) \in \tS$ with $a < b$, we let
$\alpha_{(ab)}$ (resp. $\alpha^\vee_{(ab)}$) denote the
corresponding positive root (resp. coroot), which we shall think of
as an element of the \defn{affine root lattice} $Q =
\bigoplus_{i=0}^{n-1} \Z\cdot \alpha_i$ (resp. \defn{affine coroot
lattice} $Q^\vee = \bigoplus^{n-1}_{i=0} \Z\cdot \alpha_i^\vee$). We
have $\alpha_{(ab)} = \alpha_a + \alpha_{a+1} + \cdots
\alpha_{b-1}$, where the $\alpha_i$ are the simple roots, and the
indices on the right hand side are taken modulo $n$.  A similar formula holds for coroots.
Note that $\alpha_{(ab)} = \alpha_{(a+n,b+n)}$.

In the following $s_0$ denotes $[\cdots k,k+2,k+3,\ldots,k+n-1,k+n+1 \cdots] \in \tS_n$.

\begin{lem}\label{L:affinechev}
Suppose that $f \in \Bound(k,n)$.  Then
$$
\xi^{s_0}\cdot \xi^f = \xi^{s_0}|_f \cdot \xi^f + \sum_{f \lessdot_0
g \in \Bound(k,n)} \xi^g + \; \;  \text{other terms},
$$
where the other terms are a linear combination of Schubert classes
not labeled by $\Bound(k,n)$.
\end{lem}
\begin{proof}
We deduce this formula by specializing the Chevalley formula for
Kac-Moody flag varieties in \cite{KK}\footnote{The formula in
Kostant and Kumar \cite{KK}, strictly speaking, applies to the
affine flag variety $\AFl(n)_0$ of $SL(n)$.  But each component of
$\AFl(n)$ is isomorphic to $\AFl(n)_0$.}, which in our situation
states that for any $f \in \tS$,
$$
\xi^{s_0}\cdot \xi^f = \xi^{s_0}|_f \cdot \xi^f + \sum_{f \lessdot g
= f\cdot(ab) } \langle \alpha_{(ab)}^\vee, \chi_0 \rangle\, \xi^g
$$
where $\chi_0$ is a weight of the affine root system satisfying
$\langle \alpha_i^\vee,\chi_0 \rangle = \delta_{i0}$.  We see that
$$\langle \alpha_{(ab)}^\vee,
\chi_0 \rangle = \# \left( \{\ldots,-2n,-n,0,n,2n,\ldots\} \cap [a,b)\right).$$
Now suppose that $g \in \Bound(k,n)$.  Since $i \leq g(i) \leq i +
n$, if $g\cdot (ab) \lessdot g$ then we must have $0 < b-a < n$.  In
this case, the condition that $[a,b)$ intersects
$\{\ldots,-2n,-n,0,n,2n,\ldots\}$ is the same as $f \lessdot_0 g$,
and furthermore one has $\langle \alpha_{(ab)}^\vee, \chi_0 \rangle
=1$.  This proves the Lemma.
\end{proof}

\subsection{Positroid classes and Schubert classes in affine flags}
For the subsequent discussion we work in the topological category.
Our ultimate aim is to calculate certain cohomology classes, and
changing from the algebraic to the topological category does not
alter the answers.  We refer the reader to \cite{PS, Mag} for background material.

Let $U_n$ denote the group of unitary $n \times n$ matrices and let
$T_\R \simeq (S^1)^n$ denote the subgroup of diagonal matrices. We
write $LU_n$ for the space of polynomial loops into $U_n$, and
$\Omega U_n$ for the space of polynomial based loops into $U_n$. It
is known that $LU_n/T_\R$ is weakly homotopy equivalent to
$\AFl(n)$, and that $\Omega U_n \simeq LU_n/U_n$ is weakly homotopy
equivalent to the affine Grassmannian (see \cite{PS}).

The connected components of $\Omega U_n$ and $LU_n$ are indexed by $\Z$,
using the map $L\det: LU_n \to LU_1 = Map(S^1,U(1)) \sim \pi_1(U(1)) = \Z$.
We take as our basepoint of the $k$-component of $\Omega U_n$ the loop
$t \mapsto {\rm diag}(t,\ldots,t,1,\ldots,1)$, where there are $k$
$t$'s.  Abusing notation, we write $t_{\omega_k} \in \Omega U_n$ for
this point, identifying the basepoint with a translation element.

The group $LU_n$ acts on $\Omega U_n$ by the formula
$$
(a \cdot b)(t)  = a(t)b(t)a(t)^{-1}
$$
where $a(t) \in LU_n$ and $b(t) \in \Omega U_n$.  The group $U_n$
embeds in $LU_n$ as the subgroup of constant loops. The action of
$LU_n$ on $\Omega U_n$ restricts to the conjugation
action of $U(n)$ on $U(n)$.
It then follows that the orbit of the basepoint under the action of $U_n$
\begin{equation}\label{E:orbit}
U_n \cdot t_{\omega_k} \simeq U_n/(U_k \times U_{n-k})
\end{equation}
is isomorphic to the Grassmannian $\Gr(k,n)$.

Thus we have a map $q: \Gr(k,n) \hookrightarrow \Omega U_n$.  Let
$r: \Omega U_n \to LU_n/T$ be the map obtained by composing the
natural inclusion $\Omega U_n \hookrightarrow LU_n$ with the
projection $LU_n \twoheadrightarrow LU_n/T$.  We let
$$
    p := r \circ q: \Gr(k,n) \longrightarrow LU_n/T
$$
denote the composition of $q$ and $r$.  All the maps are $T_\R$-equivariant,
so we obtain a ring homomorphism $p^*: H_T^*(\AFl(n)) \to H_T^*(\Gr(k,n))$.

\begin{lem}\label{L:fixedpoints}
  Suppose $w \in \GS$ and $I = \sigma(w)$, which we identify with a
  $T$-fixed point of $\Gr(k,n)$.
  Then $p(I) = t_{w\cdot \omega_k} T_\R \in LU_n / T_\R$.
\end{lem}

\begin{proof}
  It follows from the action of $S_n \leq U_n$ on $\Omega U_n$
  that $q(I) = t_{w \cdot \omega_k} \in \Omega U_n$.  But by definition
  $r(t_{w \cdot \omega_k}) = t_{w \cdot \omega_k} \in LU_n$.
\end{proof}

\begin{lem}\label{L:pointclass}\
\begin{enumerate}
\item Suppose $w \in \GS$.  Then $p^*(\xi^{t_{w \cdot \omega_k}}) =
[\sigma(w)]$.
\item $p^*(\xi^{s_0}) = X_\square$.
\end{enumerate}
\end{lem}
\begin{proof}
We prove (1).  Let $f = t_{w \cdot \omega_k}$.  It is enough to
check that $\xi^f|_{t_{u \cdot \omega_k}} = [\pi(w)]|_{\sigma(u)}$ for
each $u \in \GS$.  We have $[\pi(w)]|_{\sigma(u)} = 0$ unless $u = w$. By
\cite[Proposition 4.24(a)]{KK}, $\xi^f|_g = 0$ unless $f \leq g$.
Since $f$ is maximal in $\Bound(k,n)$, it is enough to calculate
$$ \xi^{f}|_f = \prod_{\alpha \in f^{-1}(\Delta_-) \cap \Delta_+} \theta(\alpha).$$
Here $\Delta_+$ (resp. $\Delta_-$) are the positive (resp. negative)
roots of the root system of $\tS_n$, and $\theta(\alpha) \in
H_T^*(\pt) = \Z[y_1,y_2,\ldots,y_n]$ denotes the image of $\alpha$
under the linear map defined by
$$
\theta(\alpha_i) = \begin{cases} y_i - y_{i+1} & \mbox{for $i = 1
,2, \ldots, n-1$,}
\\
y_n - y_1 & \mbox{for $i = 0$.}
\end{cases}
$$
Applying $\theta$ corresponds to specializing
from $H_{T \times S^1}^*(\AFl(n))$ to $H_T^*(\AFl(n))$.

With this terminology, $\alpha_{(ab)} \in f^{-1}(\Delta_-)$ if and
only if $f(a) > f(b)$.  We have $\theta(\alpha_{(ab)}) = y_a - y_b$,
where the indices are taken modulo $n$.  Thus
$$
\xi^{f}|_f = \prod_{i \in \sigma(w) \; \text{and} \; j \in [n]
\backslash \sigma(w)} (y_i - y_j)
$$
which is easily seen to agree with $[\pi(w)]|_{\sigma(w)}$.

Now we prove (2).  The class $p^*(\xi^{s_0}) \in H_T^*(G/P)$ is
of degree 2.  So by \cite[Lemma 1]{KT}, it is enough to show that it vanishes when
restricted to the identity basepoint, and equals $X_\square|_{s_k} = y_k - y_{k+1}$ when restricted to $s_k$.  We know that
$$
 \xi^{s_0}|_{\id} = 0 \qquad \text{and} \qquad \xi^{s_0}|_{s_0} = y_k - y_{k+1}
$$
since $\id < s_0$, and the inversions of $s_0^{-1}$ are exactly $\{\alpha_k\}$.  (Here $\id$ denotes $[\cdots k+1,k+2,\ldots,k+n \cdots]$.)
But we have that $t_{\omega_k}$ is in the same (right) $S_n$-coset as $\id$ and $t_{s_k \cdot \omega_k}$ is in the same $S_n$-coset as $s_0$.  Since
$\xi^{s_0}$ is a Grassmannian class, it follows that \cite{KK} $\xi^{s_0}|_{\id} = \xi^{s_0}|_{t_{\omega_k}}$ and $\xi^{s_0}|_{s_0} = \xi^{ s_0}|_{t_{s_k\cdot \omega_k}}$.  Applying Lemma \ref{L:fixedpoints}, we see that $p^*(\xi^{s_0})$ has the desired properties.

\end{proof}

\begin{theorem}\label{T:posaffineflags}
For each $f \in \tS_n^k$, we have in $H_T^*(\Gr(k,n))$,
\begin{align*}
    p^*(\xi^f) = \begin{cases} [\Pi_f] &\mbox{if $f \in \Bound(k,n)$,} \\
                     0 &\mbox{otherwise.}
    \end{cases}
\end{align*}
\end{theorem}

\begin{proof}
Suppose $f \notin \Bound(k,n)$.  Then by \cite[Proposition
4.24(a)]{KK}, $\xi^f|_g = 0$ unless $f \leq g$, so that $\xi^f|_g =
0$ for $g \in \Bound(k,n)$ (using Lemma \ref{L:orderideal}).  It
follows that $p^*(\xi^f)$ vanishes at each $T$-fixed point of
$\Gr(k,n)$, and so it is the zero class.

We shall show that the collection of classes $\{p^*(\xi^f) \mid f
\in \Bound(k,n)\}$ satisfies the conditions of Lemma
\ref{L:recursive}.  (1) is clear.  (3) follows from Lemma
\ref{L:pointclass}(1).  We check (2).  The map $p^*$ is a ring
homomorphism, so the formula in Lemma \ref{L:affinechev} holds for
the classes $p^*(\xi^f)$ as well.  Suppose $f \lessdot g \in
\Bound(k,n)$ and $f = f_{u,w}$ and $g = f_{u',w'}$.  As in the proof
of Theorem \ref{T:pairsboundedposet}, we may assume that either (1)
$u' = u$ and $w' \lessdot w$, or (2) $u' \gtrdot u$ and $w' = w$. If
$f \lessdot_0 g$, then writing $f_{u,w} = t_{u \cdot \omega_k}
uw^{-1}$ and recalling that right multiplication by $uw^{-1}$ acts
on the positions, we see that we must have $u \cdot \omega_k \neq u'
\cdot \omega_k$.  This implies that we are in Case (2), and that $u'
\gtrdot_k u$.  Conversely, if $w' = w$ and $u' \gtrdot_k u$ then we
must have $f \lessdot_0 g$. Comparing Lemma \ref{L:affinechev} and
Proposition \ref{P:monk}, and using Lemma \ref{L:pointclass}(2), we
see that we may apply Lemma \ref{L:recursive} to the classes
$\{p^*(\xi^f) \mid f \in \Bound(k,n)\}$.

Thus $p^*(\xi^f) = [\Pi_f]$ for every $f \in \Bound(k,n)$.
\end{proof}

\subsection{Affine Stanley symmetric functions}
Let $\Sym$ denote the ring of symmetric functions over $\Z$.  For
each $f \in \tS_n^0$, a symmetric function $\tF_f \in \Sym$, called
the \emph{affine Stanley symmetric function} is defined in
\cite{Lam1}.  This definition extends to all $f \in \tS_n$ via the
isomorphisms $\tS_n^k \simeq \tS_n^0$.

We will denote the simple reflections of the Coxeter group $\tS_n^0$
by $s_0,s_1,\ldots,s_{n-1}$, where the indices are taken modulo $n$.
Let $w \in \tS_n^0$.  We say that $w$ is \defn{cyclically decreasing}
if there exists a reduced expression $s_{i_1} s_{i_2} \cdots s_{i_\ell}$
for $w$ such that (a) no simple reflection is repeated, and (b) if
$s_i$ and $s_{i+1}$ both occur, then $s_{i+1}$ precedes $s_i$.  Then
the \defn{affine Stanley symmetric function} $\tF_f$ is defined by letting
the coefficient of $x_1^{a_1} x_2^{a_2} \cdots x_r^{a_r}$ in
$\tF_f(x_1,x_2,\ldots)$ to be equal to the number of factorizations
$w = w^{(1)} w^{(2)} \cdots w^{(r)}$,
where each $w^{(i)}$ is cyclically decreasing, $\ell(w^{(i)}) = a_i$,
and $\ell(w) = \ell(w^{(1)}) + \ell(w^{(2)}) + \cdots + \ell(w^{(r)})$.

For example, consider $k = 2$, $n = 4$ and $f = [5,2,7,4]$.
The corresponding element of $\tS_n^0$ is $f_0 = [3,0,5,2]$; the reduced words for $f_0$ are $s_1 s_3 s_0 s_2$, $s_1 s_3 s_2 s_0$, $s_3 s_1 s_0 s_2$ and $s_3 s_1 s_2 s_0$. So the coefficient of $x_1 x_2 x_3 x_4$ in $\tF_f$ is $4$, corresponding to these $4$ factorizations. Similar computations yield that $\tF_f = 4 m_{1111} + 2 m_{211} + m_{22} =  s_{22} + s_{211} - s_{1111}$ where the $m$'s are the monomial symmetric functions and the $s$'s are the Schur functions. Note that affine Stanley symmetric functions are not necessarily Schur positive!

The ordinary cohomology $H^*(\Omega SU_n)$ can be identified with a
quotient of the ring of symmetric functions:
$$
H^*(\Omega SU_n) \simeq \Sym/\langle m_\lambda \mid \lambda_1 > n
\rangle,
$$
where $m_\lambda$ denotes the monomial symmetric function labeled by
$\lambda$.  We refer to \cite{EC2} for general facts concerning
symmetric functions, and to \cite{Lam2} for more about $H^*(\Omega SU_n)$.

Let $s_\lambda \in \Sym$ denote the Schur functions, labeled by
partitions.  As each component of $\Omega U_n$ is homeomorphic to
$\Omega SU_n$, the inclusion $q: \Gr(k,n) \rightarrow \Omega U_n$
(defined after \eqref{E:orbit}) induces a map $\tilde \psi:
H^*(\Omega SU_n) \rightarrow H^*(\Gr(k,n))$.  Let $\psi: \Sym \to
H^*(\Gr(k,n))$ denote the composition of the quotient map $\Sym
\twoheadrightarrow H^*(\Omega SU_n)$ with $\tilde \psi: H^*(\Omega
SU_n)\rightarrow H^*(\Gr(k,n))$.

For a partition $\lambda=(\lambda_1, \lambda_2, \ldots, \lambda_k)$
with $\lambda_1 \leq n-k$,
let $\sigma(\lambda)$
be $\{ \lambda_k+1, \lambda_{k-1}+2, \ldots, \lambda_1+k \}$.
This is a bijection from partitions with at most $k$ parts and largest
part at most $n-k$ to $\binom{[n]}{k}$. We denote the set of such
partitions by $\Par(k,n)$.

\begin{lem}
The map $\psi: \Sym \to H^*(\Gr(k,n))$ is the natural quotient map
defined by
$$
\psi(s_\lambda) = \begin{cases} [X_{\sigma(\lambda)}]_0 & \mbox{$\lambda \in
\Par(k,n)$,} \\ 0 & \mbox{otherwise.} \end{cases}
$$
\end{lem}
\begin{proof} 
The copy of $\Gr(k,n)$ inside $\Omega U_n$ is the union of the
$\binom{n}{k}$ Schubert varieties labeled by the translation
elements $\{t_{w\cdot \omega_k} \mid w \in \GS\}$.  It follows that
the map $\tilde \psi: H^*(\Omega SU_n) \to H^*(\Gr(k,n))$ sends
Schubert classes to Schubert classes.

It is well known that $H^*(\Gr(k,n))$ is isomorphic to the quotient
ring of $\Sym$ as stated in the Lemma.  To check that the quotient
map agrees with $\psi$, it suffices to check that they agree on the
homogeneous symmetric functions $h_i \in \Sym$, which generate
$\Sym$.  In \cite[Theorem 7.1]{Lam2} it is shown that the Schubert
classes of $H^*(\Omega SU_n)$ are the ``affine Schur functions'',
denoted $\tF_\lambda$.  When $\lambda$ is a single row, we have
$\tF_{(r)} = h_r \in \Sym/\langle m_\lambda \mid \lambda_1 > n
\rangle$.  Furthermore, the finite-dimensional Schubert variety in
$\Omega U_n$ with dual Schubert class $\tF_{(r)}$ lies in $\Gr(k,n)
\subset \Omega U_n$ exactly when $r \leq n-k$.  It follows that
$\psi(h_r) = [X_{(r)}]_0 \in H^*(\Gr(k,n))$ for $r \leq n-k$, and
$\psi(h_r) = 0$ for $r \geq n -k$.  Thus $\psi$ is the stated map.
\end{proof}


\begin{proof}[Proof of Theorem \ref{thm:affineStanley}]  Let $\xi^f_0 \in H^*(\AFl(n))$ denote the non-equivariant Schubert
  classes.  It is shown\footnote{%
    The setup in \cite{Lam2} involves $\Omega SU_n$, but each
    component of $\Omega U_n$ is isomorphic to $\Omega SU_n$ so the
    results easily generalize.}
  in \cite[Remark 8.6]{Lam2} that we have $r^*(\xi^f_0) = \tF_f \in
  H^*(\Omega SU_n)$, where we identify $\tF_f \in \Sym$ with its image
  in $H^*(\Omega SU_n) = \Sym/\langle m_\lambda \mid \lambda_1 > n \rangle$.
  Thus we calculate using the non-equivariant version of Theorem
  \ref{T:posaffineflags}
\begin{align*}
    [\Pi_f]_0 = p^*(\xi^f_0) = q^* r^*(\xi^f_0) = \psi(\tF_f).
\end{align*}
\vskip -.3in
\end{proof}

Recall our previous example where $k=2$, $n=4$ and $f=[5274]$,
with siteswap $4040$. This positroid variety is a point.
The affine Stanley function $\tF_f$ was $s_{22} + s_{211} -
s_{1111}$, so $\psi(\tF_f) = \psi(s_{22})$, the class of a point.

\begin{example}
Stanley invented Stanley symmetric functions in order to prove that the number of reduced words for the long word $w_0$ in $S_m$ was equal to the number of standard Young tableaux of shape $(m-1, m-2,\ldots,2,1)$. 
He showed that $F_{w_0} = s_{(m-1) (m-2) \cdots 21}$, so the number of reduced words for $w_0$ is the coefficient of the monomial $m_{1 1 \cdots 1}$ in the Schur polynomial $s_{(m-1) (m-2) \cdots 21}$, as required. 
See~\cite{StanSym} for more background. 
We show how to interpret this result using positroid varieties.

Let $(k,n) = (m,2m)$.
 The Stanley symmetric function $F_{w_0}$ is the affine Stanley associated to the affine permutation 
$$v: i \mapsto \begin{cases} i+n & 1 \leq i\leq n \\ w_0(i-n) + 2n & n+1 \leq i \leq 2n \end{cases}$$
As discussed in section~\ref{sec:examples}, the positroid variety $\Pi_v$ is a graph matrix Schubert variety, and can be described as the Zariski closure, within  $G(m,2m)$, of $m$-planes that can be represented in the form
$$\mathrm{RowSpan} \begin{pmatrix} 
1 & 0 & 0 & 0 & 0 & 0 & 0 & 0 & 0 & \ast \\
0 & 1 & 0 & 0 & 0 & 0 & 0 & 0 & \ast & \ast \\
0 & 0 & 1 & 0 & 0 & 0 & 0 & \ast & \ast & \ast \\
0 & 0 & 0 & 1 & 0 & 0 & \ast & \ast & \ast & \ast \\
0 & 0 & 0 & 0 & 1 & \ast & \ast & \ast & \ast & \ast \\ \end{pmatrix}$$
 (The example shown is for $m=5$.)
 
 Reordering columns turns $\Pi_v$  into the Zariski closure of those $m$-planes that can be represented in the form
 $$\mathrm{RowSpan} \begin{pmatrix} 
 0 & 0 & 0 & 0 & 0 & 0 & 0  & 0 & 1 & \ast  \\
 0 & 0 & 0 & 0 & 0 & 0 & 1  & \ast & 0 & \ast  \\
 0 & 0 & 0 & 0 & 1 & \ast & 0  & \ast & 0 & \ast  \\
 0 & 0 & 1 & \ast & 0 & \ast & 0  & \ast & 0 & \ast  \\
  1 & \ast & 0 & \ast & 0 & \ast & 0  & \ast & 0 & \ast  \\
 \end{pmatrix}$$
 This is the Schubert variety $X_{1357\cdots (2m-1)}$, which is associated to the partition $(m-1) (m-2) \cdots 321$. 
Reordering columns acts trivially in $H^{\bullet}(G(k,n))$, so the cohomology classes of $\Pi_v$ and $X_{135\cdots (2m-1)}$ are the same,  and they thus correspond to the same symmetric function.
 This shows that $F_{w_0} = s_{(m-1) (m-2) \cdots 21}$.
 
 \end{example}

\subsection{The $K$- and $K_T$-classes of a positroid variety}
We conjecture that the $K$-class of a positroid variety is given by
the affine stable Grothendieck polynomials defined in \cite{Lam1}.
These symmetric functions were shown in \cite{LSS} to have the same
relationship with the affine flag manifold as affine Stanley
symmetric functions, with $K$-theory replacing cohomology.

\begin{conjecture}
The $K$-theory class of the structure sheaf of a positroid variety
$\Pi_f$ is given by the image of the affine stable Grothendieck
polynomial $\tilde{G}_f$, when $K^*(\Gr(k,n))$ is identified with a
ring of symmetric functions as in \cite{Buc}.
\end{conjecture}

This conjecture would follow from suitable strengthenings of
Proposition \ref{P:monk}, Lemma \ref{L:recursive}, and
Lemma \ref{L:affinechev}.
We have the necessary characterization of the $K_T$ positroid classes:
Corollary \ref{C:hypsection} and the main result of \cite{AKMobius}
give the $K_T$-analogue of Proposition \ref{P:monk}.
The degree-based argument used in Lemma \ref{L:recursive} must be
modified, in the absence of a grading on $K$-theory, to comparing
pushforwards to a point, and it is easy to show using
\cite[Theorem~4.5]{KLS} that the pushforward of a positroid class is $1$.
What is currently missing are the two corresponding results on
affine stable Grothendieck polynomials.

While the class associated to an algebraic subvariety of a
Grassmannian is always a positive combination of Schubert classes,
this is not visible from Theorem \ref{thm:affineStanley}, as affine
Stanley functions are not in general positive combinations of Schur
functions $s_\lambda$.

We can give a much stronger positivity result on positroid classes:

\begin{thm}\label{thm:positiveKTclass}
  Let $X$ be a positroid variety, and $[\O_X] \in K_T(\Gr(k,n))$ the
  class of its structure sheaf in equivariant $K$-theory.
  Then in the expansion $[\O_X] = \sum_\lambda a_\lambda [\O_\lambda]$
  into classes of Schubert varieties, the coefficient $a_\lambda \in K_T(pt)$
  lies in
  $(-1)^{|\lambda| - \dim X} \naturals\left[ \{e^{-\alpha_i}-1\} \right]$,
  where the $\{\alpha_i\}$ are the simple roots of $GL(n)$.
\end{thm}

\begin{proof}
  This is just the statement of \cite[Corollary 5.1]{AGM},
  which applies to any $T$-invariant subvariety $X$ of a flag manifold
  such that $X$ has rational singularities (as positroid varieties do,
  \cite[Corollary 4.8]{KLS}).
\end{proof}

After our first version of this preprint was circulated, 
a very direct geometric proof of Theorem \ref{T:posaffineflags} was given 
in \cite{Snider}, which also proves the corresponding statement
in equivariant $K$-theory. 
Snider identifies each affine patch on $\Gr(k,n)$ with an
opposite Bruhat cell in the affine flag manifold, $T$-equivariantly,
in a way that takes the positroid stratification to the Bruhat decomposition,
thereby corresponding the $K_T$-classes.

\section{Quantum cohomology, toric Schur functions, and positroids}\label{sec:QH}
\def\dist{\rm dist}
\def\codim{\mathrm{codim}}

\subsection{Moduli spaces of stable rational maps to the Grassmannian}
For background material on stable maps we refer the reader to~\cite{FW}.  Let $I, J \in \binom{[n]}{k}$, which we assume to be fixed throughout this section.  We now investigate the variety $S(I,J,d)$ consisting of points lying on
a stable rational map of degree $d$, intersecting $X_J \subset \Gr(k,n)$ and $X^I \subset \Gr(k,n)$.  Let $M_{0,3}(d)$ denote the moduli space of stable rational maps to $\Gr(k,n)$ with 3 marked points and degree $d$. 
Write $p_1,p_2,p_3: M_{0,3}(d) \to \Gr(k,n)$ for the evaluations at the three marked points.

Denote by $E(I,J,d)$ the subset
$$
E(I,J,d) = p_1^{-1}(X_J) \cap p_2^{-1}(X^I) \subset M_{0,3}(d).
$$
It is known \cite{FW} that $E(I,J,d)$ is reduced and locally
irreducible, with all components of dimension $\dim(X_J) + \dim(X^I)
+ dn - k(n-k)$.  Furthermore, the pushforward $(p_3)_*([E(I,J,d)])
\in H^*(\Gr(k,n))$ is a generating function for three-point, genus
zero, Gromov-Witten invariants, in the sense that
\begin{equation}\label{E:GW}
(p_3)_*([E]) \cdot \sigma = \langle [X_J], [X^I], \sigma \rangle_d
\end{equation}
for any class $\sigma \in H^*(\Gr(k,n))$. We now define $S(I,J,d) :=
p_3(E(I,J,d))$.  Let us say that \defn{there is a non-zero quantum problem
for $(I,J,d)$} if $\langle [X_J], [X^I], \sigma \rangle_d$ is
non-zero for some $\sigma \in H^*(\Gr(k,n))$.  It follows from
\eqref{E:GW} that $S(I,J,d)$ and $E(I,J,d)$ have the same dimension
whenever there is a non-zero quantum problem for $(I,J,d)$,  namely,
\begin{equation}\label{E:dimS}
  \dim(S(I,J,d)) = \dim(E(I,J,d)) = \dim(X_J) + \dim(X^I) + dn - k(n-k).
\end{equation}

The torus $T$ acts on $M_{0,3}(d)$ and, since $X_J$ and $X^I$ are
$T$-invariant, the space $E(I,J,d)$ also has a $T$-action.  The
torus fixed-points of $E(I,J,d)$ consist of maps $f: C \to
\Gr(k,n)$, where $C$ is a tree of projective lines, such that
$f_*(C)$ is a union of $T$-invariant curves in $\Gr(k,n)$ whose
marked points are $T$-fixed points, satisfying certain stability
conditions.  Since $p_3$ is $T$-equivariant, we have $S(I,J,d)^T =
p_3(E(I,J,d)^T)$.  The $T$-invariant curves in $\Gr(k,n)$ connect
pairs of $T$-fixed points labeled by $I, J \in \binom{[n]}{k}$
satisfying $|I \cap J| = k- 1$. We'll write $T(I,J,d)$ for
$S(I,J,d)^T$, considered as a subset of $\binom{[n]}{k}$.

We now survey the rest of this section. In \S \ref{sec:QComb},
we use the ideas of the previous paragraph to give an explicit
combinatorial description of $T(I,J,d)$. We then define an explicit
affine permutation $f$ associated to $(I,J,d)$ in \eqref{E:fIJd} below. We say that
$(I,J,d)$ is \defn{valid} if for $i \in I$, we have $i+k \leq f(I,J,d)(i) \leq i+n$ and, for $m \in [n] \setminus I$, we have $m \leq f(I,J,d)(m) \leq m+k$.  In particular, $(I,J,d)$ is valid implies $f(I,J,d)$ is bounded.  The main result of this
section is:

\begin{theorem} \label{thrm:QuantumPositroid}
When $(I,J,d)$ is valid, the image $p_3(E(I,J,d))$ is $\Pi_f$.
Moreover, there is one component $F_0$ of $E(I,J,d)$ for which $p_3
: F_0 \to \Pi_f$ is birational; on any other component $F$ of
$E(I,J,d)$, we have  $\dim p_3(F) < \dim F$.

When $(I,J,d)$ is not valid, then $\dim p_3(F) < \dim F$ for every
component $F$ of $E(I,J,d)$. Thus, $(I,J,d)$ is valid if and only if
there is a non-zero quantum problem for $(I,J,d)$.
\end{theorem}

Our key combinatorial result is

\begin{prop} \label{P:IJd}
Let $(I,J,d)$ be valid. Then $T(I,J,d)$ is the positroid corresponding to the bounded affine permutation $f(I,J,d)$.
\end{prop}

We should point out that we use previously known formulas for
Gromov-Witten invariants to establish part of
Theorem~\ref{thrm:QuantumPositroid}. Namely, when $f$ is valid, we
can establish directly that $p_3(E(I,J,d)) \subseteq \Pi_f$.  To
prove that $(p_3)_*([E]) = [\Pi_f]$, we combine previous work of
Postnikov with Theorem \ref{thm:affineStanley}.

It was shown in~\cite{Lam1} that $\psi(\tF_f)$ is Postnikov's ``toric
Schur function''. Postnikov showed that this toric Schur function
computed Gromov-Witten invariants but did not provide a subvariety of
$\Gr(k,n)$ representing his class; Theorem~\ref{thrm:QuantumPositroid}
can thus be viewed as a geometric explanation for toric Schur functions.


%

\subsection{Formulas for $T(I,J,d)$ and $f(I,J,d)$} \label{sec:QComb}

 We proceed to describe $S(I,J,d)^T$ explicitly.

If $I \in \binom{[n]}{k}$, then we let $A(I)$ (resp. $B(I)$) denote
the upper (resp. lower) order ideals generated by $I$.  Thus $A(J) =
\{K \in \binom{[n]}{k} \mid K \geq J\}$ is the set of $T$-fixed points
lying in $X_J$ and $B(I)$ is the set of $T$-fixed points lying in $X^I$.
Define the undirected \defn{Johnson graph} $G_{k,n}$ with vertex
set $\binom{[n]}{k}$, and edges $I \leftrightarrow J$ if $|I \cap J|
= k -1$.  The distance function $\dist(I,J)$ in $G_{k,n}$ is given
by $\dist(I,J) = k - |I \cap J|$.  Then one has
$$
S(I,J,d)^T = T(I,J,d)
:= \left\{K \in \binom{[n]}{k} \mid \dist(K,B(I)) + \dist(K,A(J))\leq d\right\}.
$$
For a pair $(I,J)$, it is shown in \cite{FW} the minimal $d$ such
that $T(I,J,d)$ is non-empty (or equivalently, that there is a path
of length $d$ from $I$ to $J$ in $G_{k,n}$), is equal to the minimal
$d$ such that a non-zero quantum problem for $(I,J,d)$ exists.

Denote by $M = \{m_1 < m_2 < \cdots < m_{n-k}\}$ the complement of $I$ in $[n]$ and similarly $L = \{l_1 < l_2 < \cdots < l_{n-k} \}$ the complement of $J$.
We define a bi-infinite sequence $\tilde{i}$ such that $\tilde{i}_a=i_a$ for $1 \leq a \leq k$ and $\tilde{i}_{a+k} = \tilde{i}_a+n$.
Similarly, we extend $J$, $M$ and $L$ to bi-infinite sequences $\tilde{j}$, $\tilde{m}$ and $\tilde{l}$, such that $\tilde{j}_{a+k} = \tilde{j}_a+n$, $\tilde{m}_{a+n-k} = \tilde{m}_a+n$ and $\tilde{l}_{a+n-k} = \tilde{l}_a+n$.
Define an affine permutation $f(I,J,D)$ by
\begin{equation}\label{E:fIJd}
f(I,J,d)(\tilde{i}_r) = \tilde{j}_{r+k-d} \quad f(I,J,d)(\tilde{m}_r) = \tilde{l}_{r+d}
\end{equation}
We say that $(I,J,d)$ is \defn{valid} if, for $i \in I$, we have $i+k \leq f(I,J,d)(i) \leq i+n$ and, for $m \in M$, we have $m \leq f(I,J,d)(m) \leq m+k$. 
In particular, if $(I,J,d)$ is valid, then $f(I,J,d)$ is bounded.

For example, let $k =2$ and $n = 6$.  Pick $I = \{1,4\}$, $J =
\{2,4\}$, $d = 1$.  Then $M = \{2,3,5,6\}$ and $L = \{1,3,5,6\}$.
The equation $f(\tilde{i}_r) = \tilde{j}_{r+k-d}$ gives $f(1) = 4$ and $f(4) = 8$.
The equation $f(\tilde{m}_r) =\tilde{l}_{r+d}$ gives $f(2) = 3$, $f(3) = 5$, $f(5) =
6$ and $f(6) = 7$.  Thus $f(I,J,d) = [\cdots 435867 \cdots]$.

Our next task is to prove proposition~\ref{P:IJd}; we shorten $f(I,J,d)$ to $f$.
Our approach is to first find the cyclic rank matrix for $f$.
Let $\tilde{I}$ and $\tilde{J}$ be the preimages of $I$ and $J$ under the projections $\ZZ \to \ZZ/n\ZZ$.
For integers $a \leq b$, we adopt the shorthand $I[a,b]$ for $\# \left( \tilde{I} \cap [a,b] \right)$ and similar notations $J[a,b)$ etcetera.
For $1 \leq a \leq b \leq a+n$, define 
\begin{equation}
r_{ab} : = \min \left( b-a+1, d+J[1,b] - I[1,a), k \right) \label{3min}
\end{equation}
and define $r_{ab}$ for all $r_{ab}$ such that $r_{a+n, b+n} = r_{ab}$; by $r_{ab} = b-a+1$ for $a>b$ and by $r_{ab} = k$ for $a+n < b$.

\begin{lem}\label{L:valid}
If $(I,J,d)$ is valid, then the matrix $r_{ab}$ is the cyclic rank matrix for $f$.
\end{lem}

It will be convenient to introduce the functions $\alpha_1(a,b) = b-a+1$, $\alpha_2(a,b) =  d+J[1,b] - I[1,a)$ and $\alpha_3(a,b) = k$, so that $r_{ab} = \min(\alpha_1(a,b), \alpha_2(a,b), \alpha_3(a,b))$.

\begin{proof}
We first check that $r_{ab}$ is a cyclic rank matrix, meaning that it obeys the conditions in Corollary~\ref{c:boundmanyfollowing}.
Conditions $(C1')$, $(C2')$ and $(C5)$ hold by definition.

For $r=1$, $2$ or $3$, it is easy to see that, $\alpha_r(a,b) - \alpha_r(a+1,b)$ and $\alpha_r(a,b+1) - \alpha_r(a,b)$ are clearly either $0$ or $1$, and are integer valued.
So $r_{ab} - r_{(a+1)b)}$ and $r_{ab} - r_{a(b-1)}$ are either $0$ or $1$.
This verifies $(C3)$.

Similarly, for any $(a,b)$ and $r=1$, $2$ or $3$, the $2 \times 2$ matrix $\left( \begin{smallmatrix} \alpha_r(a,b-1) &  \alpha_r(a,b) \\  \alpha_r(a+1 ,b-1) &  \alpha_r(a+1,b) \end{smallmatrix} \right)$ is of one of the forms
$$\begin{pmatrix} s & s \\ s & s \end{pmatrix} \quad \begin{pmatrix} s & s+1 \\ s & s+1 \end{pmatrix} \quad \begin{pmatrix} s+1 & s+1 \\ s & s \end{pmatrix} \quad \begin{pmatrix} s+1 & s+2 \\ s & s+1 \end{pmatrix}$$
for some integer $s$. Looking at what happens when we take the minimum of three matrices of this form (where we never take both the two middle matrices), we see that we always get one of the above forms, or $\left( \begin{smallmatrix} s+1 & s+1 \\ s & s+1  \end{smallmatrix} \right)$. In particular, condition $(C4)$ holds. We have now shown that $r_{ab}$ is a cyclic rank matrix. Let $g$ be the associated permutation.

We now show that $g=f$. Let $(a,b) = (i_r, \tilde{j}_{r+k-d})$ for some $r$ between $1$ and $k$.
Then $\# \left( \tilde{I} \cap [1,a) \right) = r-1$ and $\# \left( \tilde{J} \cap [1,b] \right) = r+k-d$, so $\alpha_2(a,b) = k+1$. Similar arguments let us compute
$$\begin{pmatrix} \alpha_2(a,b-1) &  \alpha_2(a,b) \\  \alpha_2(a+1 ,b-1) &  \alpha_2(a+1,b) \end{pmatrix} = \begin{pmatrix} k & k+1 \\ k-1 & k \end{pmatrix}.$$
The assumption that $i_r+k \leq f(I,J,d)(i_r) = \tilde{j}_{r+k-d}$ translates into $b-a \geq k$, so the $\alpha_1$ term in~(\ref{3min}) has no effect and we deduce that
$$\begin{pmatrix} r_{a,b-1} &  r_{a,b} \\  r_{a+1 ,b-1} &  r_{a+1,b} \end{pmatrix} = \begin{pmatrix} k & k \\ k-1 & k \end{pmatrix}.$$
By the last sentence of Corollary~\ref{c:manyfollowing}, this means that $g(a)=b$.

Similarly, if $(a,b) = (m_r, \tilde{l}_{r+d})$ we can show that $g(a)=b$.
So, for every $a \in \ZZ$, we have shown that $g(a)=f(a)$, as desired.
\end{proof}

We now begin proving the lemmas which will let us prove Proposition \ref{P:IJd}.

\begin{lem}\label{L:dist}
  We have $\dist(K,B(I)) \leq s$ if and only if for all $r \in [n]$,
  one has $I[1,r)-K[1,r) \leq s$.
\end{lem}
\begin{proof}
Suppose $\dist(K,B(I)) \leq s$.  Then $\dist(K,L) \leq s$ for some $L \leq I$.  Thus for each $r$, we have $I[1,r)-K[1,r) \leq L[1,r) - K[1,r) \leq s$.

Now suppose $I[1,r)-K[1,r) \leq s$ for each $r$.  Construct $L \leq I$ recursively, starting with $L = \emptyset$.  Assume $L \cap [1,r)$ is known.  If $r \in K$, place $r$ in $L$.  Otherwise, if $r \notin K$, place $r$ in $L$ only if $r \in I$ and $L[1,r) = I[1,r)$.  Repeat until we have constructed a $k$-element subset $L$ which clearly satisfies $L \leq I$.  The elements in $L \setminus K$ are all in $I$.
Let $\ell$ be the largest element in $L \setminus K$.  Then $I[1,\ell]$ differs from $K[1,\ell]$ by $|L \setminus K|$, and so $|L \setminus K| \leq s$.  Thus $\dist(K,L) \leq s$.
\end{proof}

\begin{lem}\label{L:TIJd}
We have $K \in T(I,J,d)$ if and only if
\begin{equation}\label{E:nomaxes}
I[1,r)-K[1,r) + K[1,s)-J[1,s) \leq d
\end{equation}
for all $1 \leq r,s, \leq n+1$.
\end{lem}
\begin{proof}
By Lemma \ref{L:dist}, we have $K \in T(I,J,d)$ if and only if
\begin{equation}\label{E:maxes}
\max(I[1,r)-K[1,r),0) + \max(K[1,s)-J[1,s),0) \leq d
\end{equation}
for all $1 \leq r,s, \leq n+1$.  Equation \eqref{E:maxes} certainly implies the stated condition.  Conversely, if \eqref{E:nomaxes} holds, but \eqref{E:maxes} fails, then we must have $I[1,r)  - K[1,r) > d$ or $K[1,s) - J[1,s) > d$ for some $r$, $s$.  In the first case setting $s = 1$ in \eqref{E:nomaxes} gives a contradiction.  In the second case, setting $r = 1$ gives a contradiction.
\end{proof}

\begin{proof}[Proof of Proposition \ref{P:IJd}]
First, suppose that $K \in T(I,J,d)$.

By Lemma \ref{L:valid}, $K$ is in the positroid corresponding to $f(I,J,d)$ if and only if for each cyclic interval $[a,b] \subset [n]$ we have
$$
K[a,b] \leq \min( b-a+1, d+J[1,b]-I[1,a), k).
$$
Since $\# [a,b] = b-a+1$ and $\#(K)=k$, we always have $K[a,b] \leq b-a+1$ and $K[a,b] \leq k$, so we must check that $K[a,b] \leq d+J[1,b] - I[1,a)$.

First, suppose that $1 \leq a \leq b \leq n$. Then $K[a,b] = K[1,b] - K[1,a)$ so the required equation is
$$K[1,b] - K[1,a) \leq  d+J[1,b] - I[1,a).$$
This is easily equivalent to~(\ref{E:nomaxes}) for $(r,s) = (a,b)$.
Now, suppose that $a \leq n < b$. Then $K[a,b] := K[1,b] - K[1,a)$ and we again want to show that $K[1,b] - K[1,a) \leq  d+J[1,b] - I[1,a)$.
Let $b=b'+n$. Then $K[1,b] = K[1,b'] + n$ and $J[1,b] = J[1,b']+n$. So it is equivalent to show 
$$K[1,b'] - K[1,a) \leq  d+J[1,b'] - I[1,a)$$
which is~(\ref{E:nomaxes}) for $(r,s) = (a,b')$.

The reverse implication is similar.
\end{proof}

\subsection{Toric shapes and toric Schur functions}\label{ssec:toric}
In \cite{PosQH}, Postnikov introduced a family of symmetric
polynomials, called {\it toric Schur polynomials}, and showed that the
expansion coefficients of these symmetric functions in terms of Schur
polynomials gave the three-point, genus zero, Gromov-Witten invariants
of the Grassmannian.  In \cite{Lam1}, it was shown that toric Schur
functions were special cases of affine Stanley symmetric functions.
We now put these results in the context of Theorem \ref{thm:affineStanley}
and equation \eqref{E:GW}: the subvariety $S(I,J,d) \subset \Gr(k,n)$ is a
positroid variety whose cohomology class is a toric Schur polynomial.

We review the notion of a toric shape and refer the reader to \cite{PosQH}
for the notion of a toric Schur function.  A \defn{cylindric shape} is
a connected, row and column convex subset of $\Z^2$ which is invariant
under the translation $(x,y) \mapsto (x+n-k,y-k)$.  Also, every row or
column of a cylindric shape must be finite, and in addition the
``border'' of a cylindric shape is an infinite path which has steps
going north and east only (when read from the southwest).
A \defn{toric shape} is a cylindric shape such that every row has at
most $n-k$ boxes, and every column has at most $k$ boxes.  For
example, the following is a toric shape for $k = 2$, $n = 5$:
$$
\tableau[sbY]{\bl&\bl&\bl&\bl&\bl&\bl&\bl&\bl&& \\
\bl&\bl&\bl&\bl&\bl&\bl&&& \\
\bl&\bl&\bl&\bl&\bl&\tf&\tf&\bl&\bl \\
\bl&\bl&\bl&\tf&\tf&\tf&\bl&\bl&\bl&\bl \\
\bl&\bl&&&\bl&\bl&\bl&\bl&\bl&\bl \\
&&&\bl&\bl&\bl&\bl&\bl&\bl&\bl
}
$$
where a fundamental domain for the action of the translation has
been highlighted.  In \cite{PosQH}, Postnikov associated a toric
shape $\theta(I,J,d)$ to each triple $(I,J,d)$ for which a
non-trivial quantum problem could be posed involving the Schubert
varieties $X^I$ and $X_J$, and rational curves of degree $d$.  The
steps of the upper border of $\theta$ is determined by $I$, the
lower border by $J$.  The gap between the two borders is determined
by $d$.  We do not give a precise description of Postnikov's
construction here as our notations differ somewhat from Postnikov's.

If $\theta$ is a cylindric shape, we can obtain an affine permutation as follows.  First label the edges of the upper border of $\theta$ by integers, increasing from southwest to northeast.  Now label the edges of the lower border of $\theta$ by integers, so that if $e$ and $e'$ are edges on the upper border and lower border respectively, and they lie on the same northwest-southeast diagonal, then $e'$ has a label which is $k$ bigger than that of $e$.  One then defines $f(\theta)$ as follows: if $a \in \Z$ labels a vertical step of the upper border, then $f(a)$ is the label of the step of the lower border on the same row; if $a \in \Z$ labels a horizontal step, then $f(a)$ is the label of the step of the lower border on the same column.  This determines $\theta(I,J,d)$ from $f(I,J,d)$ up to a translation: the equations \eqref{E:fIJd} say that the labels inside $I$ or $J$ are vertical steps, while labels in $M$ and $L$ are horizontal steps.

The condition that $(I,J,d)$ is valid translates to $\theta(I,J,d)$ being toric.  In our language, Postnikov \cite[Lemma 5.2 and Theorem 5.3]{PosQH} shows that a non-trivial quantum problem exists for $(I,J,d)$ if and only if the toric shape $\theta(I,J,d)$ is well-defined.  Thus:
\begin{lem}\label{L:qvalid}
A non-trivial quantum problem exists for $(I,J,d)$ if and only if $(I,J,d)$ is valid.
\end{lem}


\begin{lem}\label{L:boxes}
Suppose $(I,J,d)$ is valid.  Then
$$
\ell(f(I,J,d)) = |\theta(I,J,d)| = \codim(X_J) + \codim(X^I) - dn.
$$
where $ |\theta(I,J,d)|$ is the number of boxes in a fundamental domain for $\theta(I,J,d)$.
\end{lem}
\begin{proof}
The first equality follows from \cite{Lam1}, and can be explained simply as follows: each box in a fundamental domain for $\theta(I,J,d)$ corresponds to a simple generator in a reduced expression for $f(I,J,d)$.  Indeed, the equations \eqref{E:fIJd} can be obtained by filling $\theta(I,J,d)$ with a wiring diagram, where each wire goes straight down (resp. across) from a horizontal (resp. vertical) step.  The second equality follows from \cite{PosQH}.  A simple proof is as follows: if we decrease $d$ by $1$, then the lower border of $d$ is shifted one step diagonally southeast, increasing $|\theta(I,J,d)|$ by $n$.  When the upper and lower borders are far apart, then changing $\codim(X^I)$ or $\codim(X_J)$ by one also changes $|\theta(I,J,d)|$ by one.  Finally, when $I = J$ and $ d= 0$, one checks that $|\theta(I,J,d)|$ is $k(n-k)$.
\end{proof}

\begin{proof}[Proof of Theorem \ref{thrm:QuantumPositroid}]
Suppose that $(I,J,d)$ is valid.

Consider any index $K \in \binom{[n]}{k} \setminus T(I,J,d)$. Then the Pl\"ucker coordinate $p_K$ is zero on $T(I,J,d)$, and hence on $S(I,J,d)$.  By Corollary \ref{cor:pluckerdefined}, $\Pi_f$ is cut out by $\{p_K = 0 \mid K \notin T(I,J,d)\}$, so $S(I,J,d) \subseteq \Pi_f$.


By Lemma \ref{L:boxes}, \eqref{E:dimS} and Theorem \ref{thm:projectedRichardsons}, $S(I,J,d)$ and $\Pi_{f(I,J,d)}$ have the same dimension and $\Pi_f$ is irreducible. So $S(I,J,d)=\Pi_f$. Now, let $F_1$, $F_2$, \ldots, $F_r$ be the components of $E(I,J,d)$; let $c_1$, $c_2$, \ldots, $c_r$ be the degrees of the maps $p_3 : F_i \to S(I,J,d)$. Using again that $\Pi_f$ is irreducible, we know that $(p_3)_*(E(I,J,d)) = \left( \sum_{i=1}^r c_i \right) [\Pi_f]$. By the main result of~\cite{PosQH}, the left hand side of this equation is the toric Schur polynomial with shape~$\theta(I,J,d)$ and by \cite[Proposition 33]{Lam1}, this is the affine Stanley function $\psi(\tF_f)$. But by Theorem~\ref{thm:affineStanley}, the right hand side is $\left( \sum_{i=1}^r c_i \right) \psi(\tF_f)$. So $\sum_{i=1}^r c_i=1$. We deduce that $p_3$ is birational on one component of $S(I,J,d)$ and collapses every other component.

Finally, if $(I,J,d)$ is not valid, then there is no nonzero quantum product for $(I,J,d)$ by Lemma~\ref{L:qvalid}, so $p_3$ must collapse all components of $E(I,J,d)$ in this case.
\end{proof}

\subsection{Connection with two-step flag varieties}
Let $\Fl(k-d, k, k+d; n)$ and $\Fl(k-d, k+d; n)$
be the spaces of three-step and two-step flags of dimensions
$(k-d,k,k+d)$ and $(k-d,k+d)$ respectively.  We have maps $q_1 :
\Fl(k-d, k, k+d; n) \to \Gr(k,n)$ and $q_2 : \Fl(k-d, k, k+d; n) \to
\Fl(k-d, k+d; n)$.  For a subvariety $X \subset \Gr(k,n)$ we define, following \cite{BKT},
$$
X^{(d)} = q_2(q_1^{-1}(X)) \subset \Fl(k-d,k+d;n).
$$
Let us now consider the subvariety
$$
Y(I,J,d) = (X_J)^{(d)} \cap (X^I)^{(d)} \subset \Fl(k-d,k+d;n).
$$
Buch-Kresch-Tamvakis studied varieties similar to $Y(I,J,d)$, which
arise from intersections of three Schubert varieties, and showed in a
bijective manner that these intersections solved quantum problems.
Let us now consider the subvariety $q_1(q_2^{-1}(Y(I,J,d))) \subset
\Gr(k,n)$.  The subvarieties $(X_J)^{(d)}, (X^I)^{(d)} \subset
\Fl(k-d,k+d;n)$ are Schubert (and opposite Schubert) subvarieties.
Thus $q_1(q_2^{-1}(Y(I,J,d))) \subset \Gr(k,n)$ is a positroid variety
by Theorem \ref{thm:projectedRichardsons}.

The following result can also be deduced directly from \cite{BM}.

\begin{prop}
Suppose $(I,J,d)$ is valid.  Then $q_1(q_2^{-1}(Y(I,J,d))) = S(I,J,d)$.
\end{prop}

\begin{proof}
Let us first show that $S(I,J,d) \subset q_1(q_2^{-1}(Y(I,J,d)))$.  Since $(I,J,d)$ is valid, we know that $\dim S(I,J,d)= \dim(X_J) + \dim(X^I) + dn - k(n-k) =:N$. Choose $K \in \binom{[n]}{d}$ such that $\codim \ X_K=N$ and $\langle [S(I,J,d)], [X_K] \rangle \neq 0$. Then, for a general flag $F_{\bullet}$, $S(I,J,d)$ intersects $X_K(F_{\bullet})$ at a finite set of points. Moreover, the set of all points that occur as such intersections is dense in $S(I,J,d)$. (If this set were contained in a subvariety of smaller dimension, then $X_{K}(F_{\bullet})$ would miss $S(I,J,d)$ for generic $F_{\bullet}$, contradicting our choice of $K$.)

So, for $V$ in a dense subset of  $S(I,J,d)$, we know that $V$ also lies on some $X_K(F_\bullet)$ and we can impose furthermore that $F_\bullet$ is in general position with both $E_\bullet$ and $w_0 E_\bullet$.   It follows from \cite[Theorem 1]{BKT} that there is a corresponding point $W \in Y(I,J,d)$ such that $V \in q_1(q_2^{-1}(W))$.  Thus $S(I,J,d) \subset q_1(q_2^{-1}(Y(I,J,d)))$.

Conversely, let $W \in Y(I,J,d)$ be a generic point, and $Z =
q_1(q_2^{-1}(W)) \subset \Gr(k,n)$.  The space $Z$ is isomorphic to
$\Gr(d,2d)$.  Pick a point $U \in Z \cap X_J$ and $V \in Z \cap
X^I$, and another generic point $T \in Z$.  By \cite[Proposition
1]{BKT}, there is a morphism $f: \PP^1 \to Z \subset \Gr(k,n)$ of
degree $d$ which passes through $U$, $V$, and $T$.  It follows that
a generic point in $Z$ lies in $S(I,J,d)$. Thus
$q_1(q_2^{-1}(Y(I,J,d))) \subset S(I,J,d)$.
\end{proof}

\subsection{An example}
Let $k = 2$ and $n = 5$.  We take $I = J = \{1, 4\}$ and $d = 1$.
The affine permutation $f(I,J,d)$ is $[\cdots 43567 \cdots]$, with
siteswap $31222$. The positroid $T(I,J,d)$ is $\{12, 13, 14,
15,24,25,34,35,45\}$ and the juggling states are $\J(f(I,J,d)) =
(12,13,12,12,12)$.  If we pull back $Y(I,J,d)$ to $\Fl(n)$ we get
the Richardson variety $X_{12435}^{45132}$. (Following the
description given in \cite{BKT}, we obtained $12435$ by sorting the
entries of the Grassmannian permutation $14235$ in positions
$k-d+1,k-d+2,\ldots,k+d$ in increasing order.  For $45132$, we first
applied $w_0$ to $J = \{1,4\}$ to get $\{2,5\}$.  Then we did the
sorting, and left-multiplied by $w_0$ again.)

By \cite[Proposition 3.3]{KLS}, we have
$\pi(X_{12435}^{45132}) = \pi(X_{21543}^{54312})$. With $(u,w) =
(21543,54312)$, we have $f_{u,w} = [2,1,5,4,3]\cdot t_{(1,1,0,0,0)}
\cdot [4,5,3,2,1] = [4,3,5,6,7]$, agreeing with $f(I,J,d)$.
Alternatively, one can check that the $T$-fixed points inside
$X_{21543}^{54312}$, that is, the interval $[21543,54312]$, project
exactly to $T(I,J,d)$.

\end{document}